\documentclass[11pt,reqno,a4paper]{amsart}

\usepackage[utf8]{inputenc}
\usepackage[T1]{fontenc}
\usepackage{lmodern}
\usepackage[ngerman,english]{babel}
\usepackage{microtype,csquotes}

\usepackage{latexsym}
\usepackage{amsmath,amssymb,amsthm}
\usepackage{bbm}
\usepackage{mathtools}
\usepackage {enumitem}

\usepackage[backend=biber]{biblatex}  

\renewbibmacro*{publisher+location+date}{%
	\printlist{publisher}%
	\iflistundef{location}
	{\setunit*{\addcomma\space}}
	{\setunit*{\addcomma\space}}%
	\printlist{location}%
	\setunit*{\addcomma\space}%
	\usebibmacro{date}%
	\newunit}

\setlist[enumerate,1]{noitemsep}
\newtheorem{Satz}{Satz}[section]
\newtheorem{Proposition}[Satz]{Proposition} 
\newtheorem{Definition}[Satz]{Definition}     
\newtheorem{Lemma}[Satz]{Lemma}	
\newtheorem{Theorem}[Satz]{Theorem}	
\newtheorem{Corollary}[Satz]{Corollary}	
\theoremstyle{definition}

\newtheorem{Remark}[Satz]{Remark}
\newtheorem{Assumption}[Satz]{Assumption}

\numberwithin{equation}{section} 

\newcommand{\T}{\mathbb{T}} 
\newcommand{\C}{\mathbb{C}} 
\newcommand{\R}{\mathbb{R}} 
\newcommand{\Z}{\mathbb{Z}} 
\newcommand{\N}{\mathbb{N}} 
\newcommand{\eps}{\varepsilon}

\newcommand{\dd}{\, \mathrm{d}}
\newcommand{\iu}{\mathrm{i}}

\newcommand{\supp}{\operatorname{supp}}

\renewcommand{\Re}{\mathrm{Re}}

\makeatletter
\newcommand{\opnorm}{\@ifstar\@opnorms\@opnorm}
\newcommand{\@opnorms}[1]{%
	\left|\mkern-1.5mu\left|\mkern-1.5mu\left|
	#1
	\right|\mkern-1.5mu\right|\mkern-1.5mu\right|
}
\newcommand{\@opnorm}[2][]{%
	\mathopen{#1|\mkern-1.5mu#1|\mkern-1.5mu#1|}
	#2
	\mathclose{#1|\mkern-1.5mu#1|\mkern-1.5mu#1|}
}
\makeatother

\usepackage[open]{bookmark}

\usepackage{hyperref}

\allowdisplaybreaks

\title[Lie splitting for semilinear wave equations]{Error analysis of the Lie splitting for semilinear wave equations with finite-energy solutions}
\author{Maximilian Ruff}
\author{Roland Schnaubelt}
\address{Karlsruhe Institute of Technology, Department of Mathematics, Englerstraße 2, 76131 Karlsruhe, Germany}
\email{maximilian.ruff@kit.edu}
\email{schnaubelt@kit.edu}
\thanks{Funded by the Deutsche Forschungsgemeinschaft (DFG, German Research Foundation) – Project-ID 258734477 – SFB 1173.}
\thanks{We thank the referee for very interesting comments, which led to improvements of the presentation.}

\begin{document}
	\selectlanguage{english}
	\pagestyle{headings}
	\begin{abstract}
		We study time integration schemes for $\dot H^1$-solutions to the energy-(sub)critical semilinear wave equation on $\R^3$. We show first-order convergence in $L^2$ for the Lie splitting and convergence order $3/2$ for a corrected Lie splitting. To our knowledge this includes the first error analysis performed for scaling-critical dispersive problems. Our approach is based on discrete-time Strichartz estimates, including one (with a logarithmic correction) for the case of the forbidden endpoint. Our schemes and the Strichartz estimates contain frequency cut-offs.
	\end{abstract}

	\maketitle
	
	\section{Introduction}
	The semilinear wave equation $\partial_t^2u -\Delta u = \pm |u|^{\alpha-1}u$ is one of the most important model problems for dispersive behavior. Its analytical properties are well understood, see \cite{Tao2006} for a survey. In view of the energy equality, $H^1$ (or the homogeneous version $\dot H^1$) is the most natural regularity level for solutions $u(t)$ and data.
	
	On 3D-domains, in the case of powers $\alpha \in (1,3]$ one can study wellposedness by means of the standard tools of evolution equations, whereas the treatment of the case $\alpha \in (3,5]$ is based on dispersive properties. To our knowledge, in numerical analysis the latter situation has not been studied in this setting so far. (Compare \cite{RS} for the case $\alpha \in (1,3)$.) The strategy of the error analysis for such problems goes back to the seminal paper \cite{Lubich} in the case of the semilinear Schrödinger equation. In this work we establish the first error bounds for time integration schemes for the semilinear wave equation in the case $\alpha \in [3,5]$ with finite-energy data on the full space $\R^3$. 
	
	In the present paper we investigate the equation 
	\begin{equation}\begin{aligned} \label{NLW2} \partial_t^2 u &= \Delta u - \mu|u|^{\alpha-1}u, \quad (t,x) \in [0,T] \times \R^3, \\ u(0)&=u^0, \qquad \partial_tu(0)=v^0,
	\end{aligned}\end{equation}
	with finite-energy initial data $(u^0,v^0) \in \dot H^1(\R^3) \times L^2(\R^3)$, as well as parameters $\mu \in \{-1,1\}$ and $\alpha \in [3,5]$. In this setting, a complete local wellposedness theory in the space $\dot H^1(\R^3) \times L^2(\R^3)$ for $(u,\partial_tu)$ has been established, where global existence is known for the defocusing case $\mu=1$. If $\alpha>5$, the problem is ill-posed at least in the focusing case $\mu=-1$. The energy-critical case $\alpha=5$ is much more challenging than the subcritical range $\alpha<5$. The theory is based on Strichartz estimates for the linear problem, see the monographs \cite{Sogge} and \cite{Tao2006}. 
	
	To establish error bounds for time integration schemes, we show various discrete-time Strichartz estimates in Section \ref{SecStrich}. Here one controls discrete-time points $(u(n\tau))_{n\in \Z}$ of the solutions to the linear problem in spaces like $\ell^p(\Z,\ell^q(\R^3))$ by $L^2$-based norms of the initial data, where $\tau \in (0,1]$ is the time-step size. It is easy to see that a naive discrete-time version of results in continuous time fails. Instead, one has to introduce frequency cut-offs $\pi_K$ at level $K \ge 1$; {i.e., $\pi_K=\mathcal{F}^{-1}\mathbbm{1}_{\{|\xi|\le K\}} \mathcal{F}$ with the Fourier transform $\mathcal{F}$.} The estimates then depend on $K\tau$, but are otherwise in complete analogy with the estimates in continuous time. Similar results for the Schrödinger equation have been obtained in \cite{IgnatSplitting} and \cite{ORS}, see also \cite{ORSBourgain} and \cite{FilteredLie} for the case of periodic boundary conditions. Moreover, Strichartz estimates for spatially discrete Schrödinger equations were treated in the seminal works \cite{FullyDiscrete} and \cite{NumDisp}. {In contrast to the Schr\"odinger equation on $\R^n$, in the wave case one
	has to work with frequency-localized estimates and the Littlewood--Paley decomposition already for the basic Strichartz inqualities.} In Theorem \ref{ThmEndp} we also derive local-in-time estimates at the forbidden endpoint $(p,q)=(2,\infty)$ with an additional logarithmic correction depending on $K$ and the end-time $T$. Such an inequality was shown in \cite{Joly} for continuous time.
	
	The frequency cut-off has then to be introduced in the time integration schemes, too (as in the Schrödinger case, see \cite{IgnatSplitting}, \cite{ChoiKoh}, and \cite{ORS}). We first analyze the frequency-filtered Lie splitting scheme
	{\begin{equation}\label{LieFilt0}\begin{aligned}
	U_{n+1} &=  e^{\tau A}\big[U_n - \tau \big(0, \pi_{\tau^{-1}}(\mu|u_n|^{\alpha-1} u_n)\big)\big],\\
	U_0 &= (\pi_{\tau^{-1}} u^0, \pi_{\tau^{-1}}v^0)
	\end{aligned}
	\end{equation}
	for $U_n=(u_{n}, v_{n})$ and the wave operator $A(u,v)=(v,\Delta u)$,
	see \eqref{Defs} and \eqref{LieFilt}.} We show first-order convergence in $L^2(\R^3) \times \dot H^{-1}(\R^3)$ for data in $\dot H^1(\R^3) \times L^2(\R^3)$. In contrast to previous works on time discretization of semilinear wave equations in low-regularity regime, such as \cite{Gauckler}, \cite{Averaged} or \cite{KleinGordon}, in our setting there is no uniform spacetime $L^\infty$-bound on the solution $u$ available, since in three dimensions the Sobolev embedding $H^s \hookrightarrow L^\infty$ requires $s>3/2$, but we only assume $\dot H^1$ regularity of $u$. Instead, our error analysis is based on discrete-time Strichartz estimates. We use ideas from \cite{IgnatSplitting}, \cite{ChoiKoh}, and \cite{ORS}, where a similar analysis was performed in the case of the subcritical semilinear Schrödinger equation. {However, because of  the change to a second-order problem  the analysis differs in many points.}
	
	We treat the subcritical and energy-critical cases separately, since the latter requires a much more delicate analysis. 
	Our convergence result for the subcritical case $\alpha<5$ reads as follows. It is proved at the end of Section \ref{SecLie}, see also Remark \ref{RemM}.
	
	\begin{Theorem}\label{Thm1}
	Let $\alpha<5$ and $U = (u,\partial_tu) \in C([0,T],\dot H^1(\R^3)\times L^2(\R^3))$ solve the semilinear wave equation  \eqref{NLW2}. Then there are a constant $C>0$ and a maximum step size $\tau_0>0$ such that the iterates $U_n$ of the filtered Lie splitting scheme \eqref{LieFilt0} or \eqref{LieFilt} satisfy the error bound
	\[\|U(n\tau)-U_n\|_{L^2 \times \dot H^{-1}} \le C\tau \]
	for all $\tau \in (0,\tau_0]$ and $n \in \N_0$ with $n\tau \le T$. The numbers $C$ and $\tau_0$ only depend on $T$, $\alpha$, and $\|U\|_{L^\infty([0,T],\dot H^1 \times L^2(\R^3))}$.
\end{Theorem}

	 For the critical case $\alpha=5$, the next result is shown at the end of Section \ref{SecKrit}.
	
	\begin{Theorem}\label{ThmKrit}
	Let $\alpha=5$ and $U = (u,\partial_tu) \in C([0,T],\dot H^1(\R^3)\times L^2(\R^3))$ with $u \in L^4([0,T],L^{12}(\R^3))$ solve the semilinear wave equation \eqref{NLW2}. Then there are a constant $C>0$ and a maximum step size $\tau_0 >0$ such that the iterates $U_n$ of the filtered Lie splitting scheme \eqref{LieFilt} satisfy the error bound
	\[\|U(n\tau)-U_n\|_{L^2 \times \dot H^{-1}} \le C\tau \]
	for all $\tau \in (0,\tau_0]$ and $n \in \N_0$ with $n\tau \le T$. The number $C$ only depends on $T$, $\|U\|_{L^\infty([0,T],\dot H^1 \times L^2(\R^3))}$, and $\|u\|_{L^4([0,T],L^{12}(\R^3))}$, whereas $\tau_0$ only depends on $T$, $u^0$, and $v^0$.
\end{Theorem}
	The more sophisticated analysis for $\alpha=5$ is reflected by the dependence of the maximum step size $\tau_0$ on the solution itself, rather than just on its norm. A similar behavior occurs in the wellposedness theory, see \cite{Tao2006}. To our knowledge, this above theorem provides the first error analysis of a time discretization for a scaling-critical problem. 
	
	As a first step to higher-order schemes, in the case $\alpha=3$ we study a corrected Lie splitting given in \eqref{KorrLie}, {where we now take the cut-off level $K=\tau^{-3/2}$}. A version of the scheme without frequency filter was recently proposed and analyzed in \cite{KleinGordon} in higher regularity $H^{7/4}(\R^3)$. We can show convergence order $\tau^{3/2}$ for the error in $L^2(\R^3) \times \dot H^{-1}(\R^3)$, again using data in $\dot H^1(\R^3) \times L^2(\R^3)$. Formally, the scheme is of second order due to a well-chosen correction term. This term also leads to an error formula without second-order derivatives. The loss of $\tau^{1/2}$ in our result corresponds to the loss in the Strichartz estimates, so that we believe our result is optimal, compare also \cite{ORS}. In the proof we need the endpoint estimates from Theorem \ref{ThmEndp}. The convergence result is shown at the end of Section \ref{SecKorr}. {Here we use the endpoint versions of Strichartz estimates with logarithmic corrections from Theorem \ref{ThmEndp}, which seem to be new tools in the error analysis.}
	\begin{Theorem}\label{ThmKorr}
		Let $U = (u,\partial_tu) \in C([0,T],\dot H^1(\R^3)\times L^2(\R^3))$ solve the semilinear wave equation \eqref{NLW2}
		with $\alpha=3$. Then there are a constant $C>0$ and a maximum step size $\tau_0 >0$ such that the iterates $U_n$ of the corrected Lie splitting scheme \eqref{KorrLie} with $K=\tau^{-3/2}$ satisfy the error bound
		\[\|U(n\tau)-U_n\|_{L^2 \times \dot H^{-1}} \le C\tau^{\frac32} \]
		for all $\tau \in (0,\tau_0]$ and $n \in \N_0$ with $n\tau \le T$. The numbers $C$ and $\tau_0$ only depend on $T$ and $\|U\|_{L^\infty([0,T],\dot H^1 \times L^2(\R^3))}$.
	\end{Theorem}
	
	\begin{Remark}
		In the defocusing case $\mu = 1$, energy conservation shows that the solutions to \eqref{NLW2} exist globally in time. Moreover, the numbers $C$ and $\tau_0$ from Theorems \ref{Thm1} and \ref{ThmKorr}, as well as the number $C$ from Theorem \ref{ThmKrit}, then only depend on $T$, $\alpha$, $\|\nabla u^0\|_{L^2(\R^3)}$, and $\|v^0\|_{L^2(\R^3)}$. See Remarks \ref{RemDef} and \ref{RemM}.
	\end{Remark}
	
	
	We comment on variants of our results. A straightforward extension is possible to more general nonlinearities with the same growth behavior, and also to the Klein--Gordon case. Furthermore, exploiting the finite speed of propagation for the wave equation, our analysis remains valid if one replaces the spatial domain $\R^3$ by the torus $\T^3$ (where the Strichartz estimates only hold locally in time). {In this setting the above schemes become fully discrete. They are studied in more detail
	in on-going  work by the first-named author.}
	In (most of) these cases, one would have to work with inhomogeneous Sobolev spaces instead of homogeneous (dotted) ones.
	
	We do not analyze the Strang splitting (which is of second order formally) since a preliminary analysis indicates that one only obtains convergence order one for $\dot H^1$-solutions in the critical case $\alpha=5$. For $\alpha=3$, an order $\tau^{4/3}$ seems to be feasible using Strichartz estimates. However, this is inferior to the estimate for the corrected Lie splitting in Theorem \ref{ThmKorr}.
	In \cite{Gauckler} and \cite{Averaged}, second-order convergence of the Strang splitting and closely related \emph{trigonometric integrators} has been shown in case of one space dimension or $H^2$-solutions.
	
	Our analysis does not distinguish between the focusing and defocusing cases. Better results could be possible in the defocusing case, where the energy dominates the $\dot H^1$-norm. For example, an error analysis on unbounded time intervals was done in \cite{SplittingGlobal} in case of Schrödinger equations (under the additional assumption that the initial value lies in the \emph{conformal space}).\smallskip 

	In Section \ref{SecStrich} we establish the needed discrete-time Strichartz estimates, adapting ideas from the continuous time. The local wellposedness theory is recalled in Section \ref{SecWave}. We introduce the Lie splitting with frequency cut-off in Section \ref{SecLie}, and then derive the core error formula \eqref{FehlerRek}. After estimating the error terms in the subcritical case by means of our Strichartz estimates, we show Theorem \ref{Thm1} using a double induction, first iterating within a possible small time interval of size $T_1$, and then performing a recursion over intervals of length $T_1$. The critical case is studied in Section \ref{SecKrit}. Here we use in addition convergence results for $H^2$-solutions in order to make sure that the discrete approximation stays close to the PDE in $\dot H^1$. Moreover, in the argument enters how fast the Strichartz norm $\ell^4L^{12}$ of a time-discrete solution gets small on small time intervals. The last section is devoted to the corrected Lie splitting which has a more sophisticated error formula \eqref{FehlerRekKorr} and thus requires additional estimates of error terms.
	
	\medskip
	
	\textbf{Notation.} We write $A \lesssim B$ (or $A \lesssim_\alpha B$) if $A \le cB$ for a generic constant $c \ge 0$ (depending on quantities $\alpha$). Since we always work on $\R^3$, we abbreviate $L^p$ for $L^p(\R^3)$ etc. We write $\mathcal F$ for the Fourier transform, where we use the convention with the prefactor $(2\pi)^{-3/2}$. We also use the notation $\hat u \coloneqq \mathcal F u$. In the context of Fourier multipliers, we often just write $\xi$ instead of the map $\xi \mapsto \xi$. For $s \in \R$, we use the inhomogeneous and homogeneous Sobolev norms 
	\[
	\|w\|_{H^{s}} = \|(1+|\xi|^2)^{\frac{s}{2}} \hat w\|_{L^2}, \quad
	\|w\|_{\dot H^{s}} = \||\xi|^{s} \hat w\|_{L^2},
	\]
	if $\hat w \in L^1_{\mathrm{loc}}$. The homogeneous Sobolev space $\dot H^s$ is defined as
	\[\dot H^s \coloneqq \{w \in \mathcal S' : \hat w \in L^1_{\mathrm{loc}}\ \text{and}\ \|w\|_{\dot H^s} < \infty\}.\]
	This space is a Hilbert space if and only if $s<3/2$ due to Proposition 1.34 of \cite{Bahouri}. Moreover, Schwartz functions with compact Fourier support in $\R^3 \setminus \{0\}$ are dense in $\dot H^s$ if $s < 3/2$, cf.\ Proposition 1.35 of \cite{Bahouri}. 
	
	Let $p \in [1,\infty]$, $J$ be a time interval and $X$ be a Banach space. We use the Bochner spaces $L^p_JX \coloneqq L^p(J,X)$ with norms
	\[\|F\|_{L^p_JX} = \Big(\int_J\|F(t)\|_X^p\Big)^\frac{1}{p},\]
	and the usual modification for $p=\infty$. In the case $J=[0,T]$ we also write $T$ instead of $J$. If a ``free'' variable $t$ appears in such a Bochner norm, the time integration is taken with respect to $t$. Further we denote by $\ell^p \coloneqq \ell^p(\Z)$ the sequence spaces over the integers and abbreviate $ \ell^pX \coloneqq \ell^p(\Z,X)$ in case of Banach space valued sequences. In order to simplify notation we often write $\|F_n\|_{\ell^p X}$ instead of $\|(F_n)_{n \in \Z}\|_{\ell^p X}$, where again a ``free'' variable $n$ is assumed to be the summation variable. For a stepsize $\tau>0$ and a number $N \in \N_0$, we further introduce scaled norms
	\[ \|F\|_{\ell^p_\tau X} =\|F_n\|_{\ell^p_\tau X} \coloneqq \Big(\tau \sum_{n \in \Z} \|F_n\|_{X}^p \Big)^{\frac{1}{p}} \]
	and the truncated variant
	\[ \|F\|_{\ell^p_{\tau,N} X} =\|F_n\|_{\ell^p_{\tau,N} X} \coloneqq \Big(\tau \sum_{n=0}^N \|F_n\|_{X}^p \Big)^{\frac{1}{p}}. \]
	Note that in the case $p=\infty$, the norm
	$ \|F\|_{\ell^\infty_\tau X} = \sup_{n \in \Z}\|F_n\|_{X} $
	does not depend on $\tau $. For intervals $J \subseteq \R$ we also use the notation
	\[\|F\|_{\ell^p_{\tau}(J,X)} = \|F_n\|_{\ell^p_\tau(J,X)} \coloneqq \Big(\tau \sum_{\substack{n \in \Z\\ n\tau \in J}} \|F_n\|_{X}^p \Big)^{\frac{1}{p}}.  \]
	
	\section{Strichartz estimates}\label{SecStrich}
	Our analysis is based on time-discrete Strichartz estimates for the wave equation, which are established in this section.
	We start with some standard definitions and results regarding the time-continuous case.
	\begin{Definition}\label{DefLP}
		Let $\chi \in C_c^\infty(\R^3)$ be a radial function with $\chi = 1$ on $B(0,1)$, $\operatorname{supp} \chi \subseteq B(0,2)$ and 
		\[ \psi(\xi) \coloneqq \chi(\xi)-\chi(2\xi), \quad \xi \in \R^3.\]
		For every $j \in \Z$ we define
		\[\psi_j(\xi) \coloneqq \psi\Big(\frac{\xi}{2^{j}}\Big), \quad \xi \in \R^3, \quad \text{and}\quad
		P_ju \coloneqq \mathcal F^{-1}(\psi_j\hat u),\quad u \in \mathcal S'.\]
	\end{Definition}	
	These definitions yield $\operatorname{supp}\psi_j \subseteq \{\xi \in \R^3 :2^{j-1} \le  |\xi| \le 2^{j+1}\}$ and the identity
	\[\sum_{j\in\Z}\psi_j(\xi) = 1, \quad \xi \in \R^3 \setminus \{0\}. \]
	We recall that the ``Littlewood--Paley projections'' $P_j$ are bounded in $L^p$ uniformly in $j \in \Z$ and $p \in [1,\infty]$.
	
	We define the operator $|\nabla|=\sqrt{-\Delta}$ via the Fourier multiplier $|\nabla|f \coloneqq \mathcal F^{-1} (|\xi|\hat f)$, and analogously for $\cos(t|\nabla|)$ etc. To solve the wave equation, we will use the half wave group
	\[e^{\iu t |\nabla|} = \mathcal F^{-1} e^{\iu t |\xi|} \mathcal F.\]
	From Proposition III.1.5 of \cite{Sogge}, we recall the kernel bound
	\begin{equation}\label{kernelbound}
	\|\mathcal F^{-1} (e^{\iu t |\xi|}\psi)\|_{L^\infty} \lesssim (1+|t|)^{-1}, \quad t \in \R.
	\end{equation}

	The proof of the Strichartz estimates is based on the following known frequency-localized dispersive inequality. For convenience, we show how it follows from \eqref{kernelbound}.
	\begin{Lemma}\label{LemDisp}
		It holds
		\[\|P_j e^{\iu t |\nabla|} f\|_{L^q} \lesssim 2^{3j(1-\frac{2}{q})}(1+2^j|t|)^{-(1-\frac{2}{q})} \|f\|_{L^{q'}}\]
		for all $j \in \Z$, $t \in \R$, $q \in [2,\infty]$, and $f \in L^{q'}$.
	\end{Lemma}
	\begin{proof}
		Let first $f \in L^1$. Young's convolution inequality yields
		\[ \|P_j e^{\iu t |\nabla|} f\|_{L^\infty} \lesssim \|\mathcal F^{-1} (e^{\iu t |\xi|}\psi_j)\|_{L^\infty} \|f\|_{L^1},\]
		where we have
		\begin{align*}
		\mathcal F^{-1} (e^{\iu t |\xi|}\psi_j) (x) &= (2\pi)^{-3/2}\int_{\R^3} e^{\iu x \cdot \xi} e^{\iu t |\xi|}\psi(2^{-j}\xi) \dd \xi\\ &= 2^{3j} (2\pi)^{-3/2}\int_{\R^3} e^{\iu 2^jx \cdot \eta} e^{\iu 2^jt |\eta|}\psi(\eta) \dd \eta  \\
		&= 2^{3j} \mathcal F^{-1} (e^{\iu 2^jt |\xi|}\psi) (2^jx).
		\end{align*}
		Estimate \eqref{kernelbound} now gives 
		\[\|P_j e^{\iu t |\nabla|} f\|_{L^\infty} \lesssim 2^{3j}(1+2^j|t|)^{-1} \|f\|_{L^1}.\]
		The assertion follows by interpolation with the $L^2$-bound
		$ \|P_j e^{\iu t |\nabla|} f\|_{L^2} \le \|f\|_{L^2}$, which is a consequence of Plancherel's theorem.
	\end{proof}
	
	We now state some of the Strichartz estimates for the wave equation to provide a background for our results. The next theorem follows from Corollary IV.1.2 in \cite{Sogge} combined with formula \eqref{WaveGroup} below. We call a triple $(p,q,\gamma)$ \emph{wave admissible} (in dimension three) if $p \in (2,\infty]$, $q \in [2,\infty)$, and
	\begin{equation}\label{DefAdmissible}
	\frac{1}{p}+\frac{1}{q} \le \frac{1}{2}, \qquad \frac{1}{p}+\frac{3}{q} = \frac{3}{2}-\gamma. 
	\end{equation}
	One then has $\gamma \in [0,\frac32)$, and the equality in \eqref{DefAdmissible} is called \emph{scaling condition}.
	
	\begin{Theorem}\label{ThmStrich}
		Let $(p,q,\gamma)$ be wave admissible for dimension three.
		Then we have the estimates
		\begin{align*}
		\|e^{\pm \iu t |\nabla|}f\|_{L^p_\R L^q} &\lesssim_{p,q} \|f\|_{\dot H^\gamma}, \\
		\Big\|\int_{-\infty}^t e^{\pm \iu (t-s) |\nabla|}F(s) \dd s\Big\|_{L^p_\R L^q} &\lesssim_{p,q} \|F\|_{L^1_\R \dot H^\gamma} 
		\end{align*}
		for all $f \in \dot H^\gamma$ and $F \in L^1_\R\dot H^\gamma$.
	\end{Theorem}
	
	Observe that the triple $(\infty,2,0)$ corresponds to the usual energy estimate. In the above inequalities one increases the space integrability paying a price in regularity and time integrability. For the last inhomogeneous estimate, there are also variants involving $L^{\tilde p'}L^{\tilde q'}$-norms instead of the $L^1L^2$-norm. Moreover, for triples $(p,\infty,\gamma)$ with $p>2$ that satisfy \eqref{DefAdmissible}, the above estimates remain true if one replaces $L^\infty$ by $\dot B^0_{2,\infty}$. Since we will not need these facts, we omit them for simplicity. 
	
	\subsection{Time-discrete estimates}
	
	Now we turn our attention to the time-discrete setting. In the following $\ell^pL^q$ estimates, the $\ell^p$-summation is always taken over the variable $n$. We start with frequency-localized inequalities.
	
	\begin{Lemma}\label{LemLPEst}
		Let $(p,q,\gamma)$ be wave admissible. Then the estimates
		\begin{align}\label{LPEst1}
		\Big\|\sum_{k\in\Z}P_je^{\iu (n-k)|\nabla|} F_k \Big\|_{\ell^pL^q} &\lesssim_{p,q} 2^{2j\gamma}(2^{\frac{2j}{p}}+1) \|F\|_{\ell^{p'}L^{q'}}, \\\label{LPEst2}
		\Big\|\sum_{k\in\Z}P_je^{-\iu k|\nabla|} F_k \Big\|_{L^2} &\lesssim_{p,q} 2^{j\gamma}(2^{\frac{j}{p}}+1) \|F\|_{\ell^{p'}L^{q'}}, \\\label{LPEst3}
		\|P_je^{\iu n|\nabla|} f \|_{\ell^pL^q} &\lesssim_{p,q} 2^{j\gamma}(2^{\frac{j}{p}}+1) \|P_jf\|_{L^2}
		\end{align}
		hold for all $F \in \ell^{p'}L^{q'}$, $f \in L^2$, and $j \in \Z$.
	\end{Lemma}
	\begin{proof}
		We first deduce from Lemma \ref{LemDisp} the estimate
		\begin{align*}
		\Big\|\sum_{k\in\Z}P_je^{\iu (n-k)|\nabla|} F_k \Big\|_{\ell^pL^q}  &\le \Big\|\sum_{k\in\Z}\|P_je^{\iu (n-k)|\nabla|} F_k\|_{L^q} \Big\|_{\ell^p} \\ &\lesssim  2^{3j(1-\frac{2}{q})}
		\Big\|\sum_{k\in\Z} \frac{\|F_k\|_{L^{q'}}}{(1+2^j|n-k|)^{1-\frac{2}{q}}} \Big\|_{\ell^p}\\ &\le  2^{2j(\frac{1}{p}+\gamma)}
		\Big\|\sum_{k\in\Z} \frac{\|F_k\|_{L^{q'}}}{(1+2^j|n-k|)^{\frac{2}{p}}} \Big\|_{\ell^p} ,
		\end{align*}
		where the last inequality follows from the admissibility conditions \eqref{DefAdmissible}. The first assertion for $p=\infty$ is now clear. For $p < \infty$ we compute
		\begin{align*}
		2&^{2j(\frac{1}{p}+\gamma)}
		\Big\|\sum_{k\in\Z} \frac{\|F_k\|_{L^{q'}}}{(1+2^j|n-k|)^{\frac{2}{p}}} \Big\|_{\ell^p}\\
		&\le 2^{2j(\frac{1}{p}+\gamma)}\Big(\Big\| \|F_n\|_{L^{q'}}\Big\|_{\ell^p} + \Big\|\sum_{\substack{k\in\Z \\ k \neq n}} \frac{\|F_k\|_{L^{q'}}}{(1+2^j|n-k|)^{\frac{2}{p}}} \Big\|_{\ell^p}\Big) \\ &\le 2^{2j\gamma} \Big(2^{\frac{2j}{p}} \|F\|_{\ell^{p'}L^{q'}} + \Big\|\sum_{\substack{k\in\Z \\ k \neq n}} \frac{\|F_k\|_{L^{q'}}}{|n-k|^{\frac{2}{p}}} \Big\|_{\ell^p}\Big) \\ &\lesssim_{p,q} 2^{2j\gamma}(2^{\frac{2j}{p}}+1) \|F\|_{\ell^{p'}L^{q'}}
		\end{align*}
		with the help of the discrete Hardy--Littlewood--Sobolev inequality (see Proposition (a) in \cite{SteinWainger}). We note that in the case $n=k$ the factor $2^{2j/p}$ does not cancel. This is the main difference to the continuous case, where such a term does not appear in the continuous Hardy--Littlewood--Sobolev inequality. This proves \eqref{LPEst1}. 
		
		The other two claims follow by a standard $TT^*$ argument that we sketch. Using that $P_je^{-\iu k |\nabla|} = e^{-\iu k |\nabla|}P_j$, from \eqref{LPEst1} we derive
		\begin{align*}
		\Big\|\sum_{k\in\Z}P_je^{-\iu k|\nabla|} F_k \Big\|_{L^2}^2 &= \sum_{n \in \Z} \Big\langle  \sum_{k\in\Z}P_je^{\iu (n-k)|\nabla|} F_k, P_j F_n \Big\rangle \\
		&\lesssim_{p,q} 2^{2j\gamma}(2^{\frac{2j}{p}}+1) \|F\|_{\ell^{p'}L^{q'}}^2, 
		\end{align*}
		implying \eqref{LPEst2}. Here we write $\langle\cdot,\cdot\rangle$ for the $L^2$-inner product. By duality, it follows the estimate
		\begin{equation}\label{LPEst4}
		\|P_je^{\iu n|\nabla|} f \|_{\ell^pL^q} \lesssim_{p,q} 2^{j\gamma}(2^{\frac{j}{p}}+1) \|f\|_{L^2}.
		\end{equation}
		To recover $P_j$ on the right-hand side, we use the fattened Littlewood--Paley projection $\widetilde P_j\coloneqq P_{j-1}+P_j+P_{j+1}$ for $j \in \Z$, noting that $\widetilde P_j P_j = P_j$. Clearly, \eqref{LPEst4} also holds with $\widetilde P_j$ instead of $P_j$. In this inequality, we then replace $f$ by $P_jf$ to obtain the last assertion \eqref{LPEst3}.
	\end{proof}
	
	To deal with the additional factor $2^{\frac{j}{p}}+1$, we include a frequency cut-off in the discrete Strichartz estimates. For each $K \ge 1$ we define the Fourier multiplier
	\begin{equation}\label{DefPi}
	\pi_K \coloneqq \mathcal F^{-1} \mathbbm{1}_{B(0,K)}\mathcal F. 
	\end{equation}
	By Plancherel's theorem, the operators $\pi_K$ are clearly uniformly bounded in $K$ on every $L^2$-based Sobolev space. 
	
	We can now show the desired discrete Strichartz estimates. We stress that these estimates fail without the cut-off if $p<\infty$. For instance, take a function $f \in \dot H^\gamma \setminus L^q$ in \eqref{DisStrich1}. On the other hand, for $f \in \dot H^\gamma$ the map $\pi_Kf$ belongs to all $L^r$ with $r \ge q_0$ and $3/2 - \gamma = 3/q_0$ by Sobolev's embedding and Bernstein's inequality. 
	
	\begin{Theorem}\label{ThmDisStrich}
		Let $(p,q,\gamma)$ be wave admissible. Then we have the estimates
		\begin{align}\label{DisStrich1}
		\|\pi_Ke^{\iu n\tau|\nabla|} f \|_{\ell^p_\tau L^q} &\lesssim_{p,q} (K\tau)^\frac{1}{p} \|f\|_{\dot H^\gamma},\\ \label{DisStrich2}
		\Big\|\tau\sum_{k=-\infty}^{n-1}\pi_Ke^{\iu (n-k)\tau|\nabla|} F_k \Big\|_{\ell^p_\tau L^q} &\lesssim_{p,q} (K\tau)^{\frac{1}{p}}\|F\|_{\ell_\tau^{1}\dot H^{\gamma}} 
		\end{align}
		for all $\tau \in (0,1]$, $K \ge \tau^{-1}$, $f \in \dot H^\gamma$, and $F \in \ell^1 \dot H^\gamma$.
	\end{Theorem}
	
	\begin{proof}
		By approximation, it is enough to take Schwartz functions and finitely supported sequences. 
		A scaling argument reduces the estimates to the case $\tau=1$. Indeed, we can write
		\begin{equation}\label{scaling}
		\pi_Ke^{\iu t |\nabla|} f = D_{\tau^{-1}} \pi_{K\tau}e^{\iu \frac{t}{\tau} |\nabla|} D_\tau f,
		\end{equation}
		where the spatial dilation operator $D_a$ is given by $(D_a f)(x) \coloneqq f(ax)$. Assuming the case $\tau = 1$ is shown, we get the general case
		\begin{align*}
		\|\pi_Ke^{\iu n \tau |\nabla|}f\|_{\ell^p_\tau L^q} &= \tau^{\frac{1}{p}} \|D_{\tau^{-1}} \pi_{K\tau}e^{\iu n |\nabla|} D_\tau f\|_{\ell^pL^q} =\tau^{\frac{1}{p}+\frac{3}{q}} \| \pi_{K\tau}e^{\iu n |\nabla|} D_\tau f\|_{\ell^pL^q}\\ &\lesssim_{p,q} \tau^{\frac{1}{p}+\frac{3}{q}}(K\tau)^{\frac{1}{p}} \|  D_\tau f\|_{\dot H^\gamma} = \tau^{\frac{1}{p}+\frac{3}{q}-\frac{3}{2}+\gamma}(K\tau)^{\frac{1}{p}} \| f\|_{\dot H^\gamma} \\
		&= (K\tau)^{\frac{1}{p}} \| f\|_{\dot H^\gamma}
		\end{align*}
		by the scaling condition in \eqref{DefAdmissible}, and similarly for the inhomogeneous estimate.
		
		So let $\tau=1$ and $K \ge 1$. By means of the Littlewood--Paley square function estimate, Minkowski's inequality, and Lemma \ref{LemLPEst}, we compute
		\begin{align*}
		\|&\pi_Ke^{\iu n |\nabla|}f\|_{\ell^pL^q} \\
		&\lesssim_q \Big\|\Big(\sum_{j\in\Z}|P_j\pi_Ke^{\iu n |\nabla|}f|^2\Big)^\frac{1}{2}\Big\|_{\ell^pL^q} \le \Big(\sum_{j\in\Z}\|P_je^{\iu n |\nabla|}\pi_Kf\|^2_{\ell^pL^q}\Big)^\frac{1}{2} \\ &\lesssim_{p,q}
		\Big(\sum_{j\in\Z}\|2^{j\gamma}(2^{\frac{j}{p}}+1) P_j\pi_Kf\|^2_{L^2}\Big)^\frac{1}{2} \lesssim K^{\frac{1}{p}}\Big(\sum_{j\in\Z}\|2^{j\gamma} P_jf\|^2_{L^2}\Big)^\frac{1}{2}\\
		&\lesssim K^{\frac{1}{p}}\|f\|_{\dot H^\gamma},
		\end{align*}
		also using that $ P_j \pi_K=0$ for $K \lesssim 2^j$. Thus, the homogeneous estimate \eqref{DisStrich1} is true. By duality, we infer the dual homogeneous estimate
		\begin{equation}\label{DualHom}
		\Big\|\sum_{k \in \Z} \pi_K e^{-ik|\nabla|}G_k\Big\|_{\dot H^{-\gamma}} \lesssim_{p,q} K^{\frac{1}{p}} \|G\|_{\ell^{p'}L^{q'}}
		\end{equation}
		that is valid for all $G \in \ell^{p'}L^{q'}$.
		
		The truncated inhomogeneous estimate \eqref{DisStrich2} is proven by another duality argument via
		\begin{align*}
		\Big\|&\sum_{k=-\infty}^{n-1}\pi_Ke^{\iu (n-k)|\nabla|} F_k \Big\|_{\ell^p L^q} \\
		&= \sup_{\|G\|_{\ell^{p'}L^{q'}}\le 1}\Big| \sum_{n \in \Z} \Big\langle \sum_{k=-\infty}^{n-1}\pi_Ke^{\iu (n-k)|\nabla|} F_k,G_n \Big\rangle \Big| \\
		&= \sup_{\|G\|_{\ell^{p'}L^{q'}}\le 1}\Big| \sum_{k \in \Z} \Big\langle  F_k,\sum_{n=k+1}^{\infty}\pi_Ke^{\iu (k-n)|\nabla|}G_n \Big\rangle \Big| \\
		&\le \|F\|_{\ell^1\dot H^\gamma} \sup_{\|G\|_{\ell^{p'}L^{q'}}\le 1} \sup_{k \in \Z}\Big\|\sum_{n=k+1}^{\infty}\pi_Ke^{\iu (k-n)|\nabla|}G_n\Big\|_{\dot H^{-\gamma}}.
		\end{align*}
		Using \eqref{DualHom}, the assertion now follows from
		\begin{align*}
		\sup_{\|G\|_{\ell^{p'}L^{q'}}\le 1} &\sup_{k \in \Z}\Big\|\sum_{n=k+1}^{\infty}\pi_Ke^{\iu (k-n)|\nabla|}G_n\Big\|_{\dot H^{-\gamma}} \\
		&= \sup_{\|G\|_{\ell^{p'}L^{q'}}\le 1} \sup_{k \in \Z}\Big\|\sum_{n\in\Z}\pi_Ke^{-\iu n|\nabla|}\mathbbm{1}_{\{n \ge k+1\}}G_n\Big\|_{\dot H^{-\gamma}}\\
		&\lesssim_{p,q} K^{\frac{1}{p}}\sup_{\|G\|_{\ell^{p'}L^{q'}}\le 1} \sup_{k \in \Z}\|\mathbbm{1}_{\{n \ge k+1\}}G_n\|_{\ell^{p'}L^{q'}} = K^{\frac{1}{p}}.
		\end{align*}
		Alternatively, one could also employ the Christ--Kiselev Lemma \ref{LemChristKiselev} below to deduce the inhomogeneous estimate from the homogeneous one.
	\end{proof}
	
	We discuss some variants of the above results and proofs.
	
	\begin{Remark}\label{RemDisStr}
		a) Using the Bernstein inequality $\|\pi_Kf\|_{\dot H^\gamma} \lesssim K^\gamma \|f\|_{L^2}$ that holds for every $L^2$-function $f$ and $\gamma \ge 0$, we can convert the derivative loss from the discrete Strichartz estimates into a multiplicative factor of the form $K^\gamma$. This gives e.g.
		\begin{align*}
		\|\pi_Ke^{\iu n\tau|\nabla|} f \|_{\ell^p_\tau L^q} &\lesssim_{p,q} (K\tau)^\frac{1}{p}K^\gamma \|f\|_{L^2},
		\end{align*}
		and similarly for the inhomogeneous estimates. \smallskip
		
		b) In all the previous estimates we can replace the plus sign in the exponential by a minus sign (e.g., $e^{-\iu n\tau|\nabla|}$ instead of $e^{\iu n\tau|\nabla|}$) since we can instead of $f$ resp.\ $F$ always use their complex conjugates and $\mathbbm1_{B(0,K)}(\xi) = \mathbbm1_{B(0,K)}(-\xi)$. This modification is employed below without further notice.\smallskip
		
		c) In our applications, we will always deal with finite sequences, defined on some set $\{0,\dots,N\}$ for an integer $N\in \N_0$. In that case, the second estimate of Theorem \ref{ThmDisStrich} takes the form
		\[ 		\Big\|\tau\sum_{k=0}^{n-1}\pi_Ke^{\iu (n-k)\tau|\nabla|} F_k \Big\|_{\ell^p_{\tau,N} L^q} \lesssim_{p,q} (K\tau)^{\frac{1}{p}}\|F\|_{\ell_{\tau,N-1}^{1}\dot H^{\gamma}}, \]
		where on the right-hand side we only need to consider the index range $\{0,\dots,N-1\}$.
	\end{Remark}
	
	\begin{Remark}
		There exists an alternative (simpler) approach to time-discrete Strichartz estimates, which uses the well-known continuous estimates just as a ``black box''. In the context of Schrödinger equations, it was used in Lemma 2.6 of the recent preprint \cite{WuModified}, see also Lemma 2.1 of \cite{TaoBilin} for a similar technique. But this approach yields a weaker estimate compared to Theorem \ref{ThmDisStrich} in the case when $K>\tau^{-1}$. We give the details. For technical reasons, here we have to replace the frequency cut-off $\pi_K$ by a version $\widetilde \pi_K$ with a smooth cut-off function (similar as the Littlewood--Paley projections from Definition \ref{DefLP}). Let first $\tau=1$ and $K>0$ be arbitrary. For a function $f \in \dot H^\gamma$, we get
		\begin{align*}
		\|&e^{\iu n |\nabla|}\widetilde \pi_Kf\|^p_{\ell^pL^q} \\
		&= \sum_{n \in \Z} \int_{n-1}^n \|e^{\iu n |\nabla|}\widetilde \pi_Kf\|^p_{L^q} \dd t \\ 
		&\lesssim_p \sum_{n \in \Z} \int_{n-1}^n \|(e^{\iu n |\nabla|}-e^{\iu t |\nabla|})\widetilde \pi_Kf\|^p_{L^q} \dd t + \sum_{n \in \Z} \int_{n-1}^n\|e^{\iu t |\nabla|}\widetilde \pi_Kf\|^p_{L^q} \dd t.
		\end{align*}
		Note that the last term is equal to $\|e^{\iu t |\nabla|}\widetilde \pi_Kf\|_{L^pL^q}^p$, therefore it can be treated directly by the continuous Strichartz estimate from Theorem \ref{ThmStrich}. The first term can be rewritten as
		\begin{align*}
		\sum_{n \in \Z} \int_{n-1}^n \|(e^{\iu n |\nabla|}-e^{\iu t |\nabla|})\widetilde \pi_Kf\|^p_{L^q} \dd t &= \sum_{n \in \Z} \int_{n-1}^n \Big\|\int_t^n \iu |\nabla| e^{\iu \sigma |\nabla|}\widetilde \pi_Kf \dd \sigma\Big\|^p_{L^q} \dd t \\
		&\le \sum_{n \in \Z} \int_{n-1}^n \int_{n-1}^n  \||\nabla| e^{\iu \sigma |\nabla|}\widetilde \pi_Kf \|^p_{L^q}\dd \sigma \dd t \\
		&\lesssim_{p,q} K^p \sum_{n \in \Z}  \int_{n-1}^n  \| e^{\iu \sigma |\nabla|}f\|^p_{L^q} \dd \sigma, 
		\end{align*}
		where we used Bernstein's inequality and the frequency cut-off $\widetilde \pi_K$ to get rid of the differential operator. In this step it would be necessary to use $\widetilde \pi_K$ with a smooth cut-off function. Now we are again in a position to apply Theorem \ref{ThmStrich}. Altogether, this gives the estimate
		\[\|\widetilde \pi_Ke^{\iu n |\nabla|}f\|_{\ell^pL^q} \lesssim_{p,q} (1+K)\|f\|_{\dot H^\gamma}. \]
		A scaling argument as in the proof of Theorem \ref{ThmDisStrich} then yields the estimate 
		\[\|\widetilde \pi_Ke^{\iu n \tau |\nabla|}f\|_{\ell^p_\tau L^q} \lesssim_{p,q} (1+K\tau)\|f\|_{\dot H^\gamma}, \]
		for general $\tau>0$.
		We see that this estimate is inferior to Theorem \ref{ThmDisStrich} if $K>\tau^{-1}$, but for $K=\tau^{-1}$ they are the same.
	\end{Remark}
	
	\subsection{Endpoint estimates with logarithmic loss}
	
	The Strichartz estimates from Theorem \ref{ThmStrich} and \ref{ThmDisStrich} are in general wrong if $(p,q,\gamma)=(2,\infty,1)$, see \cite{TaoBilin} or Exercise 2.44 in \cite{Tao2006} for a discussion. In this section we show a local-in-time estimate with logarithmic loss, which will be used to discuss the corrected Lie splitting in Section \ref{SecKorr}. We follow the approach from Section 8 in \cite{Joly} and transfer it to the time-discrete setting. First, we need two lemmas with basic estimates. The first one is contained in the proof of Lemma 8.1 in \cite{Joly}.
	
	\begin{Lemma}\label{LemM}
		The function
		\[M \colon \R \times \R^3 \to \C, \quad M(\lambda,z) \coloneqq \int_{B(0,1)} |\xi|^{-2} \cos(\lambda|\xi|)e^{\iu z \cdot \xi} \dd \xi \]
		satisfies the decay estimate
		\[|M(\lambda,z)| \lesssim \frac{1}{1+|\lambda|} \]
		for all $\lambda \in \R$ and $z \in \R^3$.
	\end{Lemma}
	
	\begin{Lemma}\label{LemA}
		The function
		\[A \colon \R \times \N_0 \times \N_0 \to \R, \quad A(\beta,n,j) \coloneqq  \Big(\frac{1}{1+\beta|n-j|} + \frac{1}{1+\beta(n+j)}\Big)\]
		satisfies the inequality
		\[\max_{j=0,\dots,N} \sum_{n=0}^N A(\beta,n,j) \lesssim 1+\beta^{-1}\log(1+N\beta) \]
		for all $N \in \N_0$ and $\beta>0$.
	\end{Lemma}
	\begin{proof}
		Let $j \in \{0,\dots,N\}$. We have
		\begin{align*}
		\sum_{n=0}^N  \frac{1}{1+\beta(n+j)} &\le \sum_{n=0}^N  \frac{1}{1+\beta n} = 1+ \sum_{n=1}^N  \frac{1}{1+\beta n} \le 1+\beta^{-1} \int_0^{N\beta}\frac{1}{1+t} \dd t \\
		&= 1+\beta^{-1}\log(1+N\beta)
		\end{align*}
		and
		\begin{align*}
		\sum_{n=0}^N  \frac{1}{1+\beta|n-j|} &= 1 + \sum_{n=0}^{j-1}  \frac{1}{1+\beta(j-n)} + \sum_{n=j+1}^{N} \frac{1}{1+\beta(n-j)} \\
		&= 1 + \sum_{n=1}^{j}  \frac{1}{1+\beta n} + \sum_{n=1}^{N-j} \frac{1}{1+\beta n} \le 1 + 2\sum_{n=1}^{N} \frac{1}{1+\beta n}\\
		&\le 1 + 2\beta^{-1}\log(1+N\beta). \qedhere
		\end{align*}
	\end{proof}
	
	Now we show the announced time-discrete endpoint estimates with logarithmic loss.
	
	\begin{Theorem}\label{ThmEndp}
		The estimates
		\begin{align}\label{Endp1}
		\|\pi_Ke^{\iu n\tau|\nabla|} f \|_{\ell^2_{\tau,N} L^\infty} &\lesssim \sqrt{K\tau + \log(1+KN\tau )} \|f\|_{\dot H^1}, \\ \label{EndpDual}
		\Big\|\tau\sum_{k=0}^N\pi_Ke^{-\iu k\tau|\nabla|} F_k \Big\|_{\dot H^{-1}} &\lesssim \sqrt{K\tau + \log(1+KN\tau )} \|F\|_{\ell^2_{\tau,N}L^1}
		\end{align}
		hold for all $\tau \in (0,1]$, $K \ge \tau^{-1}$, $N \in \N_0$, $f \in \dot H^1$, and $F \in \ell^2L^1$.
	\end{Theorem}
	\begin{proof}
		By duality, the two estimates are equivalent. Due to a scaling argument, for \eqref{Endp1} it is enough to show 
		\begin{equation}\label{EndpktAlpha}
		\|\pi_1e^{\iu n\beta|\nabla|} f \|_{\ell^2_{1,N} L^\infty} \lesssim \sqrt{1 + \beta^{-1}\log(1+N\beta)} \|f\|_{\dot H^1}
		\end{equation}
		for any $\beta>0$. Indeed, this estimate and \eqref{scaling} imply
		\begin{align*}
		\|\pi_Ke^{\iu n \tau |\nabla|}f\|_{\ell^2_{\tau,N}L^\infty}  &= \tau^{\frac12} \|D_K \pi_1e^{\iu n K\tau |\nabla|} D_{K^{-1}}f\|_{\ell^2_{1,N}L^\infty} \\
		&= \tau^{\frac12} \| \pi_1e^{\iu n K\tau |\nabla|} D_{K^{-1}}f\|_{\ell^2_{1,N}L^\infty}\\
		&\lesssim \tau^{\frac12}\sqrt{1 + (K\tau)^{-1}\log(1+NK\tau)} \|D_{K^{-1}}f\|_{\dot H^1} \\
		&=  (K\tau)^{\frac12}\sqrt{1 + (K\tau)^{-1}\log(1+NK\tau)} \|f\|_{\dot H^1} \\
		&=\sqrt{K\tau + \log(1+NK\tau)} \|f\|_{\dot H^1}.
		\end{align*}
		
		We show \eqref{EndpktAlpha} via the dual estimate
		\begin{equation}\label{EndpktAlphaDual}
		\Big\|\sum_{n=0}^N\pi_1e^{-\iu n\beta|\nabla|} F_n \Big\|_{\dot H^{-1}} \lesssim \sqrt{1 + \beta^{-1}\log(1+N\beta)} \|F\|_{\ell^2_{1,N}L^1}
		\end{equation}
		for $F \in \ell^2L^1$. Instead of the exponential, we treat sine and cosine. From the definition of the $\dot H^1$-norm and Fubini's theorem, we deduce
		\begin{align*}
		\Big\|&\sum_{n=0}^N\pi_1\sin(n\beta|\nabla|) F_n \Big\|_{\dot H^{-1}}^2 =  \Big\||\xi|^{-1}\sum_{n=0}^N\mathbbm{1}_{B(0,1)}\sin(n\beta|\xi|) \hat F_n(\xi) \Big\|_{L^2}^2 \\
		&= \int_{B(0,1)} |\xi|^{-2}\sum_{n,j=0}^N\sin(n\beta|\xi|)\sin(j\beta|\xi|) \hat F_n(\xi)\overline {\hat F_j(\xi)} \dd \xi \\
		&= (2\pi)^{-3} \int_{B(0,1)} |\xi|^{-2}\sum_{n,j=0}^N\sin(n\beta|\xi|)\sin(j\beta|\xi|) \\
		&\quad \cdot \int_{\R^3} \int_{\R^3} e^{\iu(y-x)\cdot \xi}F_n(x)\overline {F_j(y)} \dd x \dd y \dd \xi \\
		&= (2\pi)^{-3}  \sum_{n,j=0}^N \int_{\R^3} \int_{\R^3} G_-(n\beta,j\beta,y-x) F_n(x)\overline {F_j(y)} \dd x \dd y,
		\end{align*}
		where
		\begin{equation*}
		G_-(a,b,z) \coloneqq \int_{B(0,1)} |\xi|^{-2}\sin(a|\xi|)\sin(b|\xi|)e^{\iu z\cdot \xi}  \dd \xi.
		\end{equation*}
		Analogously, we obtain
		\begin{align*}
		\Big\|&\sum_{n=0}^N\pi_1\cos(n\beta|\nabla|) F_n \Big\|_{\dot H^{-1}}^2 \\
		&=  
		(2\pi)^{-3}  \sum_{n,j=0}^N \int_{\R^3} \int_{\R^3} G_+(n\beta,j\beta,y-x) F_n(x)\overline {F_j(y)} \dd x \dd y,
		\end{align*}
		with
		\begin{equation*}
		G_+(a,b,z) \coloneqq \int_{B(0,1)} |\xi|^{-2}\cos(a|\xi|)\cos(b|\xi|)e^{\iu z\cdot \xi}  \dd \xi.
		\end{equation*}
		Trigonometric identities 
		yield
		\begin{equation*}
		G_{\pm}(a,b,z) = \frac12 \Big(M(a-b,z) \pm M(a+b,z) \Big)
		\end{equation*}
		with $M(\lambda,z)$ from Lemma \ref{LemM}. Combined with this lemma, the above equations lead to
		\begin{align*}
		\Big\|&\sum_{n=0}^N\pi_1e^{-\iu n\beta|\nabla|} F_n \Big\|_{\dot H^{-1}}^2 \\
		&\lesssim \sum_{n,j=0}^N \int_{\R^3} \int_{\R^3} \Big(|M(\beta(n-j),y-x)|+|M(\beta(n+j),y-x)|\Big)  \\
		&\qquad \qquad \qquad \quad \cdot |F_n(x)||F_j(y)|\dd x \dd y \\
		&\lesssim \sum_{n,j=0}^N \int_{\R^3} \int_{\R^3} \Big(\frac{1}{1+\beta|n-j|}+\frac{1}{1+\beta|n+j|}\Big) |F_n(x)||F_j(y)| \dd x \dd y \\
		&= \sum_{n,j=0}^N  A(\beta,n,j) \|F_n\|_{L^1}\|F_j\|_{L^1},
		\end{align*}
		with $A$ from Lemma \ref{LemA}. We can now apply Cauchy--Schwarz twice and Lemma \ref{LemA}. Also noting that $A(\beta,n,j)$ is symmetric in $n$ and $j$, we estimate
		\begin{align*}
		\sum_{n,j=0}^N  A(&\beta,n,j) \|F_n\|_{L^1}\|F_j\|_{L^1}\\
		&\le \|F\|_{\ell^2_{1,N}L^1} \Big[\sum_{n=0}^N\Big(\sum_{j=0}^NA(\beta,j,n)\|F_j\|_{L^1}\Big)^2\Big]^{\frac12} \\
		&\le \|F\|_{\ell^2_{1,N}L^1}\Big[\sum_{n=0}^N\Big(\sum_{j=0}^NA(\beta,j,n)\Big)\Big(\sum_{j=0}^NA(\beta,n,j)\|F_j\|_{L^1}^2\Big)\Big]^{\frac12} \\
		&\le \Big(\max_{n=0,\dots,N}\sum_{j=0}^NA(\beta,n,j)\Big)^{\frac12}\Big(\max_{j=0,\dots,N}\sum_{n=0}^NA(\beta,n,j)\Big)^{\frac12}\|F\|_{\ell^2_{1,N}L^1}^2 \\
		&\lesssim \Big(1+\beta^{-1}\log(1+N\beta)\Big)\|F\|_{\ell^2_{1,N}L^1}^2,
		\end{align*} 
		which shows \eqref{EndpktAlphaDual}.
	\end{proof}
	
	
	In Section \ref{SecKorr} we also need an inhomogeneous estimate with the forbidden exponents $(\tilde p', \tilde q',1) = (2,1,1)$ on the right-hand side. Estimates with $\tilde q \neq 2$ can often be deduced by means of the \emph{Christ--Kislev Lemma} from \cite{ChristKiselev}.  
	
	\begin{Lemma}\label{LemChristKiselev}
		Let $X$ and $Y$ be measure spaces, and $E$ and $F$ be Banach spaces, $1 \le p < q \le \infty$, and $T \colon L^p(X,E) \to L^q(Y,F)$ be a bounded linear operator. Furthermore, let $(X_j)_{j\in \N}$ be a sequence of measurable subsets of $X$ such that $X_j \subseteq X_{j+1}$ for all $j \in \N$. We define the \emph{Christ--Kiselev maximal operator} 
		\[ [T^*h](y) \coloneqq \sup_{j \in \N} \|[T(h\mathbbm{1}_{X_j})](y)\|_F \]
		for all $h \in L^p(X,E)$ and almost every $y \in Y$. Then we have the estimate
		\[\|T^*h\|_{L^q(Y)} \le (1-2^{-(p^{-1}-q^{-1})})^{-1}\|T\| \|h\|_{L^p(X,E)} \]
		for all $h \in L^p(X,E)$.
	\end{Lemma}
	
	We show the estimate to be used in Section \ref{SecKorr}. There we need uniformity with respect to shifts $s \in \R$ in the wave propagator.
	
	\begin{Corollary}\label{KorInhomCK}
		Let $(p,q,\gamma)$ be wave admissible. Then the inequality
		\begin{align*}
		\Big\|\tau&\sum_{k=0}^{n-1} \pi_Ke^{\iu(n-k+s)\tau|\nabla|} |\nabla|^{-1} F_k \Big\|_{\ell^p_{\tau,N}L^q} \\ &\lesssim_{p,q} (K\tau)^{\frac1p}K^\gamma\sqrt{K\tau+\log(KN\tau)}\|F\|_{\ell^2_{\tau,N-1}L^1}  
		\end{align*}
		holds for all $\tau \in (0,1]$, $K \ge \tau^{-1}$, $s \in \R$, $N \in \N$, and $F \in \ell^2L^1$.
	\end{Corollary}
	\begin{proof}
		We define the sets $X_0 \coloneqq \emptyset$ and $X_j \coloneqq \{0,\dots,\min\{j-1,N\}\}$ for $j \in \N$, and $X \coloneqq\{0,\dots,N\}$. By the Christ--Kiselev Lemma \ref{LemChristKiselev}, it suffices to prove the estimate for the operator given by
		\[ (TF)_n \coloneqq \tau\sum_{k=0}^N \pi_Ke^{\iu(n-k+s)\tau|\nabla|} |\nabla|^{-1} F_k.  \]
		We compute
		\begin{align*}
		\Big\|&\tau\sum_{k=0}^N \pi_Ke^{\iu(n-k+s)\tau|\nabla|} |\nabla|^{-1} F_k \Big\|_{\ell^p_{\tau,N}L^q} \\ &\lesssim_{p,q} (K\tau)^{\frac1p} \Big\|\tau\sum_{k=0}^N \pi_Ke^{\iu(-k+s)\tau|\nabla|} |\nabla|^{-1} F_k \Big\|_{\dot H^\gamma} \\
		&\le (K\tau)^{\frac1p}K^\gamma \Big\|\tau\sum_{k=0}^N \pi_Ke^{-\iu k\tau|\nabla|}  F_k \Big\|_{\dot H^{-1}} \\
		&\lesssim (K\tau)^{\frac1p}K^\gamma\sqrt{K\tau + \log(1+KN\tau )}\|F\|_{\ell^2_{\tau,N}L^1}.
		\end{align*}
		Here we first use the homogeneous Strichartz estimate \eqref{DisStrich1}. Then we apply Bernstein's inequality to get rid of the derivative loss of order $\gamma$, at the cost of the factor $K^\gamma$. In the end, we employ the endpoint estimate \eqref{EndpDual} from Theorem \ref{ThmEndp}. 
	\end{proof}
	
	Furthermore, we need the following ``hybrid'' estimates to control time-continuous solutions in time-discrete norms. Their proof is only sketched since it does not require new ideas.
	
	\begin{Corollary}\label{KorHybStrich}
		Let $(p,q,\gamma)$ be wave admissible. Then the estimates
		\begin{align}\label{Hyb1}
		\Big\|\int_{-\infty}^{n\tau}\pi_Ke^{\iu (n\tau-s)|\nabla|} F(s) \dd s \Big\|_{\ell^p_\tau L^q} &\lesssim_{p,q} (K\tau)^{\frac{1}{p}}\|F\|_{L^1_\R\dot H^{\gamma}}, \\ \label{Hyb2}
		\Big\|\int_{0}^{n\tau}\pi_Ke^{\iu (n\tau-s)|\nabla|} F(s) \dd s \Big\|_{\ell^2_{\tau,N} L^\infty} &\lesssim \sqrt{K\tau + \log(1+KN\tau )}\|F\|_{L^{1}([0,N\tau],\dot H^1)} 
		\end{align}
		hold for all  $\tau \in (0,1]$, $K \ge \tau^{-1}$, $F \in L^1_\R \dot H^\gamma$, and $N \in \N_0$.
	\end{Corollary}
	\begin{proof}
		Two proofs are possible. Since we have an $L^2$-based Sobolev space in the space variable on the right-hand side of the inequality, we can apply the same technique as in the proof of \eqref{DisStrich2} in Theorem \ref{ThmDisStrich}, just replacing one of the two sums by an integral. For \eqref{Hyb2} one uses \eqref{EndpDual} instead of \eqref{DualHom}. An alternative way would be to proceed as in the previous proof and use the Christ--Kiselev Lemma with the sets $X_j \coloneqq (-\infty,j\tau] \subseteq \R$ for the first, resp.\ $X_j \coloneqq [0,\min\{j,N\}\tau] \subseteq [0,N\tau]$ for the second estimate. So \eqref{Hyb1} is reduced to \eqref{DisStrich1}, and \eqref{Hyb2} to \eqref{Endp1}.
	\end{proof}
	
	\section{Nonlinear Wave equation}\label{SecWave}
	We introduce some notation and recall some well-known results from the wellposedness theory of the nonlinear wave equation \eqref{NLW2}. For the nonlinearity we write $g(u) \coloneqq -\mu|u|^{\alpha-1}u$. We will often use the elementary pointwise Lipschitz bound
	\begin{equation}\label{NichtlinLip}
	|g(v)-g(w)| \lesssim (|v|^{\alpha-1}+|w|^{\alpha-1})|v-w|
	\end{equation}
	for $g$.
	It is convenient to reformulate \eqref{NLW2} as a system having first order in time. Therefore we set
	\begin{equation}\label{Defs}
	U \coloneqq \begin{pmatrix}u \\ v\end{pmatrix} \mathrel{\widehat{=}} \begin{pmatrix}u \\ \partial_t u\end{pmatrix},\ A \coloneqq \begin{pmatrix}0 & I \\ \Delta & 0 \end{pmatrix},\ G(U) \coloneqq \begin{pmatrix}0 \\ g(u)\end{pmatrix},\ U^0 \coloneqq \begin{pmatrix}u^0 \\ v^0\end{pmatrix},
	\end{equation}
	and obtain the equivalent first-order system
	\begin{equation}
	\begin{aligned} \label{NLW} \partial_{t} U(t) &= AU(t) +G(U(t)), \quad t \in [0,T], \\ U(0)&=U^0.
	\end{aligned}\end{equation}
	We are looking for a mild solution of \eqref{NLW}, i.e., a function $U \in C([0,T],\dot H^1\times L^2)$ satisfying the Duhamel formula
	\begin{equation}\label{Duh}
	U(t) = e^{t A }U^0+\int_0^t e^{(t-s)A}G(U(s)) \dd s
	\end{equation}
	for all $t \in [0,T]$. The group $e^{tA}$ is given by
	\begin{equation}\label{WaveGroup}
	e^{tA} = \begin{pmatrix}\cos(t|\nabla|) & |\nabla|^{-1} \sin(t|\nabla|) \\ -|\nabla| \sin(t|\nabla|) & \cos(t|\nabla|) \end{pmatrix},\quad t \in \R. \end{equation}
	Inserting \eqref{WaveGroup}, the first line of \eqref{Duh} reads
	\begin{equation}\label{Duh1}
	u(t) = \cos(t|\nabla|)u^0 + |\nabla|^{-1} \sin(t|\nabla|)v^0 + \int_0^t |\nabla|^{-1} \sin((t-s)|\nabla|)g(u(s)) \dd s.
	\end{equation}
	For the linear part of the evolution we abbreviate 
	\begin{equation}\label{Duh1lin}
	S(t)(u^0,v^0) \coloneqq \cos(t|\nabla|)u^0 + |\nabla|^{-1} \sin(t|\nabla|)v^0.
	\end{equation}
	In this setting we can apply the Strichartz estimates as formulated in Section \ref{SecStrich}, since we can decompose
	\[\sin(t|\nabla|) = \frac{1}{2\iu}(e^{\iu t |\nabla|}-e^{-\iu t |\nabla|}), \qquad \cos(t|\nabla|) = \frac12(e^{\iu t |\nabla|}+e^{-\iu t |\nabla|}).\]
	In the case $\alpha = 3$, Sobolev embedding shows that the nonlinearity $G$ leaves the space $\dot H^1 \times L^2$ invariant. Therefore, local wellposedness can be shown in a standard way using the Duhamel formula and Banach's fixed point theorem. If $\alpha >3$, a more sophisticated analysis is needed since in that case the nonlinearity loses too much integrability. By means of Strichartz estimates, local wellposedness can be shown up to the critical power $\alpha=5$. 
	
	Let $\alpha \in [3,5]$. We define an exponent $p_\alpha \in [4,\infty]$ such that $(p_\alpha,3(\alpha-1),1)$ are wave admissible, i.e., by the relation
	\[\frac{1}{p_\alpha}+\frac{1}{\alpha-1} = \frac{1}{2}, \]
	see \eqref{DefAdmissible}.
	The local wellposedness theorem for the nonlinear wave equation \eqref{NLW2} reads as follows. A proof for the critical case $\alpha=5$ is given in, e.g., Theorem IV.3.1 of \cite{Sogge}, and the subcritical case $\alpha<5$ can be proven similarly. See also Sections 8.3--8.4 of \cite{Bahouri} for related results and Chapter 3 in \cite{Tao2006} for an overview.
	\begin{Theorem}\label{ThmLokWoh}
		Let $U^0 \in \dot H^1 \times L^2$. Then there exists a time $T_0>0$ and a unique solution $ U = (u,\partial_tu)$ of \eqref{NLW2} such that $U \in C([0,T_0], \dot H^1 \times L^2)$ and $u \in L^{p_\alpha}([0,T_0],L^{3(\alpha-1)})$.
		Moreover, $U$ is bounded in the above function spaces by a constant depending only on $\|U^0\|_{\dot H^1 \times L^2}$, and in the subcritical case $\alpha<5$, the time $T_0$ only depends on $\|U^0\|_{\dot H^1 \times L^2}$.
	\end{Theorem}
	
	\begin{Remark}\label{RemDef}
		a) Since $g(u) \in C([0,T_0],L^{6/\alpha}) \hookrightarrow C([0,T_0],H^{-1})$ and $\Delta u \in C([0,T_0],\dot H^{-1}) \hookrightarrow C([0,T_0], H^{-1})$, one can deduce from \eqref{Duh1} that $\partial_t^2 u$ belongs to $C([0,T_0], H^{-1})$ and that the differential equation in \eqref{NLW2} holds in this space. \smallskip
		
		b) See Proposition \ref{PropLokStab} for a choice of $T_0$ in the critical case $\alpha=5$. \smallskip
		
		c) In the defocusing case $\mu=1$, one can show global-in-time wellposedness using energy conservation. In the focusing case $\mu=-1$ however, blow-up in finite time can occur. Moreover, if $\alpha>5$, the problem is ill-posed at least if $\mu=-1$. See Chapters 3 and 5 in \cite{Tao2006}. \smallskip
		
		d) In the subcritical case $\alpha<5$, one has uniqueness of solutions $U = (u,\partial_tu)$ to \eqref{NLW2} in the energy class $C([0,T_0], \dot H^1 \times L^2)$ without the requirement that $u \in L^{p_\alpha}([0,T_0],L^{3(\alpha-1)})$, cf.\ \cite{PlanchonUniqueness}. 
	\end{Remark}
	
	From now on we assume the following. 
	
	\begin{Assumption}\label{Ass}
		There exists a time $T>0$ and a solution $U = (u,\partial_tu)$ of \eqref{NLW} such that
		\[U \in C([0,T], \dot H^1 \times L^2)\quad \text{and}\quad u \in L^{p_\alpha}([0,T],L^{3(\alpha-1)})\]
		with bound
		\begin{equation}\label{DefM}
		M \coloneqq \max\{\|U\|_{L^\infty_T(\dot H^1 \times L^2)}, \|u\|_{L^{p_\alpha}_TL^{3(\alpha-1)}} \}.
		\end{equation}
	\end{Assumption}
	
	\begin{Remark}\label{RemM}
		If $\alpha<5$, the quantity $M$ only depends on $\|U\|_{L^\infty_T(\dot H^1 \times L^2)}$. Indeed, the ``minimal'' existence time $T_0$ and the number $M$ for $T=T_0$ are controlled by $\|U^0\|_{\dot H^1 \times L^2}$ in Theorem \ref{ThmLokWoh}. Hence, we can divide the interval $[0,T]$ into a finite number of smaller subintervals such that the Strichartz norm of $u$ is bounded on each of them. Moreover, in the defocusing case $\mu=1$, one can use energy conservation to show that the number $M$ only depends on $\|U^0\|_{\dot H^1 \times L^2}$. The latter even holds in the critical case $\alpha=5$, see \cite{TaoKritBound} and the references therein.
	\end{Remark}
	
	As a preparation for later sections, we now want to convert the continuous-time Strichartz estimate into a discrete-time space-time bound for $u$ from Assumption \ref{Ass}. To apply the results of Section \ref{SecStrich}, we need to include the frequency cut-off $\pi_K$ defined by \eqref{DefPi}.

	\begin{Proposition}\label{PropDiskrStrichU}
		Let $u$, $T$, and $M$ be given by Assumption \ref{Ass} and $(p,q,1)$ be wave admissible. Then we have the estimates
		\begin{align*}
		\|\pi_K u(n\tau +\sigma)\|_{\ell^p_{\tau,N}L^q} &\lesssim_{p,q,M,T} (K\tau)^{\frac1p},  \\
		\|\pi_K u(n\tau +\sigma)\|_{\ell^2_{\tau,N}L^\infty} &\lesssim_{M,T} (K\tau +\log K)^{\frac12}
		\end{align*}
		for all $\tau \in (0,1]$, $K \ge \tau^{-1}$, $\sigma \ge 0$, and $N \in \N$ with $N\tau + \sigma \le T$.
	\end{Proposition}
	\begin{proof}
		Owing to \eqref{Duh1} and \eqref{Duh1lin}, the solution is given by
		\begin{align*}
		&\pi_K u(n\tau +\sigma) \\
		&= \pi_KS(n\tau)(u(\sigma),\partial_t u(\sigma)) +\int_0^{n\tau}\pi_K  |\nabla|^{-1} \sin((n\tau-s)|\nabla|)g(u(\sigma+s)) \dd s .
		\end{align*}
		We apply the homogeneous and hybrid Strichartz estimates \eqref{DisStrich1} and \eqref{Hyb1} to obtain
		\begin{align*}
		\|\pi_K &u(n\tau +\sigma)\|_{\ell^p_{\tau,N}L^q} \\ &\lesssim_{p,q} (K\tau)^{\frac1p} \Big(\|u(\sigma)\|_{\dot H^1} + \|\partial_t u(\sigma)\|_{L^2} + \|g(u(\sigma+\cdot))\|_{L^1_{T-\sigma}L^2}\Big)\\ &\lesssim (K\tau)^{\frac1p} \Big( M+\||u|^{\alpha-1}u\|_{L^1_TL^2}\Big)\\
		&\le (K\tau)^{\frac1p}\Big(M + \||u|^{\alpha-1}\|_{L^1_TL^3}\|u\|_{L^\infty_TL^6}\Big) \\ &\lesssim
		(K\tau)^{\frac1p}\Big(M + T^{2-\frac{\alpha-1}{2}} \|u\|^{\alpha-1}_{L^{p_\alpha}_TL^{3(\alpha-1)}} M\Big) \lesssim_{M,T} (K\tau)^{\frac1p},
		\end{align*}
		where we also use Hölder's inequality and the Sobolev embedding $\dot H^1 \hookrightarrow L^6$ to bound the nonlinearity. The estimate in the $\ell^2L^\infty$-norm follows in the same way, using the logarithmic endpoint estimates \eqref{Endp1} and \eqref{Hyb2}.
	\end{proof}
	
	\begin{Remark}\label{RemStrichU}
		With a similar calculation we can also show the continuous-time bound
		\[\|\pi_K u\|_{L^p_TL^q} \lesssim_{p,q,M,T} 1. \]
		Here we use the continuous-time Strichartz estimate from Theorem \ref{ThmStrich} and the uniform boundedness of the operator $\pi_K$ in all $L^2$-based Sobolev spaces. 
	\end{Remark}
	
	\section{Lie splitting}\label{SecLie}
	We consider a semi-discretization in time for \eqref{NLW}, i.e., for a stepsize $\tau >0$ and $n \in \N$ we want to compute approximations $U_n = (u_n,v_n) \approx U(t_n)$ at discrete time points $t_n \coloneqq n\tau$. One can consider the following Lie Splitting scheme (sometimes also called Lawson--Euler scheme)
	\begin{equation*}
	U_{n+1} = e^{\tau A}[U_n + \tau G(U_n)]
	\end{equation*}
	with $A$ and $G$ from \eqref{Defs}.
	In order to perform one step of the Lie splitting scheme, we first apply one step of the exact nonlinear flow, followed by one step of the exact linear flow. One can also derive the Lie Splitting by setting $s=0$ in the integral of the Duhamel formula
	\begin{equation*}\label{Duhtau}
	U(\tau) = e^{\tau A }U^0+\int_0^\tau e^{(\tau-s)A}G(U(s)) \dd s.
	\end{equation*}
	To compensate the ``bad'' behavior of the nonlinearity $G$ in the case $\alpha>3$, we analyze a modified scheme that includes the frequency cut-off $\pi_K$ defined in \eqref{DefPi}, where we choose $K = \tau^{-1}$. (One could also take $K = c\tau^{-1}$ for some fixed $c>0$.) Thanks to this choice, the factor $(K\tau)^{1/p}$ in the discrete Strichartz estimates vanishes. 
	We introduce the notation
	\[ \Pi_{\tau^{-1}} \coloneqq \begin{pmatrix}\pi_{\tau^{-1}} & 0 \\ 0 & \pi_{\tau^{-1}} \end{pmatrix}. \]  
	The modified scheme is given by
	\begin{equation}\label{LieFilt}\begin{aligned}
	U_{n+1} &= \Phi_{\tau}(U_n) \coloneqq e^{\tau A}[U_n + \tau\Pi_{\tau^{-1}} G(U_n)],\\
	U_0 &= \Pi_{\tau^{-1}} U^0.
	\end{aligned}
	\end{equation}
	The filter is applied to the initial data, as well as in the iteration after the application of the nonlinearity. Inductively, one checks that $\Pi_{\tau^{-1}} U_n = U_n$ for all $n$. Let the solution $U=(u,\partial_tu)$ be given by Assumption \ref{Ass}. In Lemma \ref{LemPi} we control the projection error $U(t_n) - \Pi_{\tau^{-1}}U(t_n)$. Hence, to compare the numerical approximation with $U$, we define the (main) error
	\begin{equation}\label{DefMainError}
	E_n \coloneqq \Pi_{\tau^{-1}} U(t_n) - U_n
	\end{equation}	
	for all $n \in \N_0$ with $n\tau = t_n \le T$. Note that with this definition we have $E_0 = 0$, since the numerical scheme filters the initial data. We write $u_n$ for the first component of $U_n$, as well as $e_n$ for the first component of $E_n$. We first establish a recursion formula for the error.
	
	\begin{Proposition}\label{PropFehlerRek}
		Let $E_n$ be given by \eqref{DefMainError} and \eqref{LieFilt} for the solution $U$ from Assumption \ref{Ass}. We then have
		\begin{align}\label{FehlerRek}
		E_{m+n} &= e^{n\tau A}E_m + \int_0^{t_n}\Pi_{\tau^{-1}} e^{(n\tau-s)A}B_m(s) \dd s \nonumber\\
		&\quad +  \sum_{k=0}^{n-1}\Pi_{\tau^{-1}} e^{(n-k)\tau A}\Big(\Delta_{m+k} + \tau Q_{m+k}\Big) 
		\end{align}
		for all $\tau \in (0,1]$, and $n, m \in \N_0$ with $t_{m+n} \le T$. The appearing terms are given by 
		\begin{align}\label{DefTerme}
		B_m(s) &\coloneqq G(U(t_m+s))-G(\Pi_{\tau^{-1}} U(t_m+s)),\nonumber\\
		\Delta_n &\coloneqq \int_0^\tau e^{-sA} G(\Pi_{\tau^{-1}} U(t_n+s)) \dd s -\tau G(\Pi_{\tau^{-1}} U(t_n)),\\
		Q_n &\coloneqq G(\Pi_{\tau^{-1}} U(t_n))-G(U_n). \nonumber
		\end{align}
	\end{Proposition}
	\begin{proof}
		We use the Duhamel formula 
		\begin{equation*}
		U(t_{m+n}) = e^{n\tau A} U(t_m) + \int_0^{t_n} e^{(n\tau-s)A}G(U(t_m+s)) \dd s
		\end{equation*}
		for the solution $U$. For the discrete approximation $U_n$ defined by \eqref{LieFilt} we have the discrete Duhamel formula
		\begin{equation}\label{DiskrDuhamelMitM}
		U_{m+n} = e^{n\tau A}U_m + \tau \sum_{k=0}^{n-1}\Pi_{\tau^{-1}} e^{(n-k)\tau A} G(U_{m+k}),
		\end{equation}
		which is easily verified via induction. These equalities yields
		\begin{align*}
		E_{m+n} &=\Pi_{\tau^{-1}} U(t_{m+n}) - U_{m+n}\\
		&=  e^{n\tau A} (\Pi_{\tau^{-1}} U(t_m)-U_m) + \int_0^{t_n} \Pi_{\tau^{-1}} e^{(n\tau-s)A}G(U(t_m+s)) \dd s\\
		&\quad - \tau \sum_{k=0}^{n-1}\Pi_{\tau^{-1}} e^{(n-k)\tau A} G(U_{m+k})\\
		&= e^{n\tau A} E_m + \int_0^{t_n} \Pi_{\tau^{-1}} e^{(n\tau-s)A} B_m(s) \dd s \\
		&\quad + \int_0^{t_n}\Pi_{\tau^{-1}} e^{(n\tau-s)A}G(\Pi_{\tau^{-1}} U(t_m+s)) \dd s\\
		&\quad - \tau \sum_{k=0}^{n-1}\Pi_{\tau^{-1}} e^{(n-k)\tau A} G(\Pi_{\tau^{-1}} U(t_{m+k})) + \tau \sum_{k=0}^{n-1}\Pi_{\tau^{-1}} e^{(n-k)\tau A} Q_{m+k}.
		\end{align*}
		Furthermore, we can rewrite the last integral as
		\begin{align*}
		\int_0^{t_n}\Pi_{\tau^{-1}} &e^{(n\tau-s)A}G(\Pi_{\tau^{-1}} U(t_m+s))\dd s \\
		&= \sum_{k=0}^{n-1}\int_{t_k}^{t_{k+1}}\Pi_{\tau^{-1}} e^{(n\tau-s)A}G(\Pi_{\tau^{-1}} U(t_m+s))\dd s \\
		&= \sum_{k=0}^{n-1}\Pi_{\tau^{-1}} e^{(n-k)\tau A}\int_{0}^{\tau}e^{-sA}G(\Pi_{\tau^{-1}} U(t_{m+k}+s))\dd s
		\end{align*}
		to obtain the expression including the local error terms $\Delta_{m+k}$.
	\end{proof}
	
	We now quantify convergence $\pi_K \to I$ as $K \to \infty$. In view of later applications, the following two results are formulated for general $K \ge 1$. In this section, we will only use the case $K = \tau^{-1}$. We give the proof of the next (known) lemma for convenience.
	\begin{Lemma}\label{LemPi}
		The estimate
		\begin{align*}
		\|(I-\pi_K)f\|_{\dot H^{\gamma-1}} &\le \frac1K\|f\|_{\dot H^{\gamma}} 
		\end{align*}
		is true for all $K \ge 1$, $\gamma \in \R$, and $f \in \dot H^\gamma$. 
	\end{Lemma}
	\begin{proof}
		We simply use the support property of the Fourier multiplier $\pi_K$, Plancherel's theorem and the definition of the Sobolev norm. This gives
		\begin{align*}
		\|(I-\pi_K )f\|_{\dot H^{\gamma-1}} &= \||\xi|^{\gamma-1}(1-\mathbbm{1}_{B(0,K)})\hat f\|_{L^2} \le \Big\|\frac{\mathbbm{1}_{\{|\xi|\ge K\}}}{|\xi|}\Big\|_{L^\infty}\||\xi|^\gamma \hat f\|_{L^2} \\
		&\le \frac1K\|f\|_{\dot H^{\gamma}}. \qedhere
		\end{align*} 
	\end{proof}
	Using this, we show an estimate for the terms $B_n(s)$ in \eqref{DefTerme}.
	\begin{Lemma}\label{LemB}
		Let $u$, $T$, and $M$ be given by Assumption \ref{Ass}. We then have the inequality
		\[\|g(u)-g(\pi_K u) \|_{L^1_T \dot H^{-1}} \lesssim_{M,T} \frac1K \]
		for all $K \ge 1$.
	\end{Lemma}
	\begin{proof}
		We use estimate \eqref{NichtlinLip}, the Sobolev embedding $L^{6/5}\hookrightarrow\dot H^{-1}  $, Hölder's inequality, the definition \eqref{DefM} of $M$, Remark \ref{RemStrichU}, and Lemma \ref{LemPi} to obtain
		\begin{align*}
		\|&g(u)-g(\pi_K u) \|_{L^1_T \dot H^{-1}} \\
		&\lesssim \|(|u|^{\alpha-1}+|\pi_K u|^{\alpha-1})|u-\pi_K u|\|_{L^1_T L^{6/5}} \\
		&\le(\|u\|^{\alpha-1}_{L_T^{\alpha-1}L^{3(\alpha-1)}}+\|\pi_K u\|^{\alpha-1}_{L_T^{\alpha-1}L^{3(\alpha-1)}})\|u-\pi_K u\|_{L^\infty_T L^2}\\
		&\lesssim_{M,T} \frac1K \|u\|^{\alpha-1}_{L_T^{p_\alpha}L^{3(\alpha-1)}}\|u\|_{L^\infty_T \dot H^1} \lesssim_M \frac1K. \qedhere
		\end{align*}
	\end{proof}
	
	We next bound the local error terms $\Delta_n$ from \eqref{DefTerme}. To differentiate $g$, we identify $\C$ with $\R^2$ using the real scalar product $z \cdot w = \Re(z\bar w)$, where we omit the dot below.
	
	\begin{Lemma}\label{LemLokF}
		Let $U = (u,\partial_tu)$, $T$, and $M$ be given by Assumption \ref{Ass}. We then have the representation
		\begin{equation}\label{ErrRepres}
		\Delta_n = \int_0^\tau \int_0^s e^{-\sigma A}\begin{pmatrix}-g(\pi_{\tau^{-1}} u(t_n+\sigma)) \\ g'(\pi_{\tau^{-1}} u(t_n+\sigma))\pi_{\tau^{-1}} \partial_t u(t_n+\sigma)\end{pmatrix} \dd \sigma\dd s. 
		\end{equation}
		Moreover, the inequality
		\[\|\Delta_n\|_{\ell^1_{\tau,N}(L^2\times \dot H^{-1})} \lesssim_{M,T} \tau^2 \]
		holds for all $\tau \in (0,1]$ and $N \in \N_0$ with $(N+1)\tau \le T$.
	\end{Lemma}
	\begin{proof}
		Starting from \eqref{DefTerme}, we write
		\begin{align*}
		\Delta_n &= \int_0^\tau \Big[e^{-sA}G(\Pi_{\tau^{-1}} U(t_n+s))  - G(\Pi_{\tau^{-1}} U(t_n))\Big]\dd s \\
		&= \int_0^\tau \int_0^s \frac{\mathrm{d}}{\mathrm{d}\sigma} \Big[e^{-\sigma A}G(\Pi_{\tau^{-1}} U(t_n+\sigma)) \Big]\dd \sigma\dd s \\
		&= \int_0^\tau \int_0^s e^{-\sigma A}\Bigg[-\begin{pmatrix}0 & I \\ \Delta & 0 \end{pmatrix}\begin{pmatrix}0 \\ g(\pi_{\tau^{-1}} u(t_n+\sigma))\end{pmatrix}\\
		&\qquad \qquad \qquad \quad+\frac{\mathrm{d}}{\mathrm{d}\sigma}\begin{pmatrix}0 \\ g(\pi_{\tau^{-1}} u(t_n+\sigma))\end{pmatrix}\Bigg] \dd \sigma\dd s\\
		&= \int_0^\tau \int_0^s e^{-\sigma A}\begin{pmatrix}-g(\pi_{\tau^{-1}} u(t_n+\sigma)) \\ g'(\pi_{\tau^{-1}} u(t_n+\sigma))\pi_{\tau^{-1}} \partial_t u(t_n+\sigma)\end{pmatrix} \dd \sigma\dd s.
		\end{align*}
		Thanks to the regularization $\pi_{\tau^{-1}}$ there are no problems taking the derivative. We can now estimate
		\begin{align*}
		\|&\Delta_n\|_{\ell^1_{\tau,N}(L^2\times \dot H^{-1})} \\
		&\lesssim \tau^2 \sup_{\sigma \in [0,\tau]} \Big\|\begin{pmatrix}g(\pi_{\tau^{-1}} u(t_n+\sigma)) \\ g'(\pi_{\tau^{-1}} u(t_n+\sigma))\pi_{\tau^{-1}} \partial_t u(t_n+\sigma)\end{pmatrix}\Big\|_{\ell^1_{\tau,N}(L^2\times L^{6/5})} \\ &\lesssim  \tau^2 \sup_{\sigma \in [0,\tau]} \||\pi_{\tau^{-1}} u(t_n+\sigma)|^{\alpha-1}\|_{\ell^1_{\tau,N}L^3}\big(\|\pi_{\tau^{-1}} u\|_{L^\infty_T L^6} + \|\pi_{\tau^{-1}} \partial_t u\|_{L^\infty_T L^2}\big) \\ &\lesssim_T \tau^2 \sup_{\sigma \in [0,\tau]} \| \pi_{\tau^{-1}} u(t_n+\sigma)\|^{\alpha-1}_{\ell^{p_\alpha}_{\tau,N}L^{3(\alpha-1)}}\big(\| u\|_{L^\infty_T \dot H^1} + \| \partial_t u\|_{L^\infty_T L^2}\big) \\ &\lesssim_{M,T} \tau^2,
		\end{align*}
		using Sobolev's embedding, the estimate $|g'(w)z| \lesssim |w|^{\alpha-1}|z|$, Hölder's inequality, and Proposition \ref{PropDiskrStrichU}.
	\end{proof}
	
	Now we turn our attention to the terms $Q_n$ defined in \eqref{DefTerme}. To estimate these terms, we will need an a priori bound on the numerical solution $u_n$ in the discrete Strichartz norm $\|\cdot\|_{\ell^4_{\tau,N}L^{3(\alpha-1)}}$. Since such a bound is at first unclear, we use the relation $u_{n} = \pi_{\tau^{-1}} u(t_{n})-e_{n}$ combined with the discrete Strichartz estimate for $u$ from Proposition \ref{PropDiskrStrichU}. Hence, for proving the global error bound in Theorem \ref{Thm1}, we must in addition show the convergence $\|e_n\|_{\ell^4_{\tau,N}L^{3(\alpha-1)}} \to 0$ as $\tau \to 0$, see Proposition \ref{PropKonvSubkrit}. This idea goes back to \cite{Lubich} (in a setting with maximum norms in time).
	
	\begin{Lemma}\label{LemStab}
		Let $u$, $T$, and $M$ be given by Assumption \ref{Ass}. Then the inequality
		\begin{align*}
		\|&g(\pi_{\tau^{-1}} u(t_{m+n}))-g(u_{m+n})\|_{\ell^1_{\tau,j}\dot H^{-1}} \\
		&\lesssim_{M,T}  t_{j+1}^{1-\frac{\alpha-1}{4}}\Big(1+\|e_{m+n}\|^{\alpha-1}_{\ell^4_{\tau,j}L^{3(\alpha-1)}}\Big)\|e_{m+n}\|_{\ell^\infty_{\tau,j}L^2} 
		\end{align*}
		holds for all $\tau \in (0,1]$ and $m, j \in \N_0$ with $(m+j)\tau \le T$. 
	\end{Lemma}
	\begin{proof}
		Similar as in Lemma \ref{LemB} we estimate
		\begin{align*}
		\|&g(\pi_{\tau^{-1}} u(t_{m+n}))-g(u_{m+n})\|_{\ell^1_{\tau,j}\dot H^{-1}}\\ &\lesssim \Big(\|\pi_{\tau^{-1}} u(t_{m+n})\|^{\alpha-1}_{\ell^{\alpha-1}_{\tau,j}L^{3(\alpha-1)}}+\|u_{m+n}\|^{\alpha-1}_{\ell^{\alpha-1}_{\tau,j}L^{3(\alpha-1)}}\Big)\\
		&\quad \cdot \|\pi_{\tau^{-1}} u(t_{m+n})-u_{m+n}\|_{\ell^\infty_{\tau,j}L^2}\\
		&\le t_{j+1}^{1-\frac{\alpha-1}{4}}\Big(\|\pi_{\tau^{-1}} u(t_{m+n})\|^{\alpha-1}_{\ell^4_{\tau,j}L^{3(\alpha-1)}}+\|u_{m+n}\|^{\alpha-1}_{\ell^4_{\tau,j}L^{3(\alpha-1)}}\Big)\|e_{m+n}\|_{\ell^\infty_{\tau,j}L^2}\\
		&\lesssim_T t_{j+1}^{1-\frac{\alpha-1}{4}}\Big(\|\pi_{\tau^{-1}} u(t_{m+n})\|^{\alpha-1}_{\ell^{p_\alpha}_{\tau,j}L^{3(\alpha-1)}}+\|e_{m+n}\|^{\alpha-1}_{\ell^4_{\tau,j}L^{3(\alpha-1)}}\Big)\|e_{m+n}\|_{\ell^\infty_{\tau,j}L^2}\\
		&\lesssim_{M,T}  t_{j+1}^{1-\frac{\alpha-1}{4}}\Big(1+\|e_{m+n}\|^{\alpha-1}_{\ell^4_{\tau,j}L^{3(\alpha-1)}}\Big)\|e_{m+n}\|_{\ell^\infty_{\tau,j}L^2}.
		\end{align*}
		Here we keep the power of $t_{j+1}$ gained by Hölder from the change from the $\ell^{\alpha-1}_{\tau,j}$- to the $\ell^4_{\tau,j}$-norm and we further insert $u_{m+n} = \pi_{\tau^{-1}} u(t_{m+n})-e_{m+n}$. The estimate in the last line follows from Proposition \ref{PropDiskrStrichU}.
	\end{proof}
	
	Due to scaling considerations, the convergence order of $\|e_n\|_{\ell^4_{\tau,N}L^{3(\alpha-1)}}$ will be $\frac{1}{\alpha-1}-\frac{1}{4}$ since we only assume $\dot H^1$ regularity of the solution $u$. Thus, in the critical case $\alpha=5$ we cannot prove any convergence rate in this Strichartz norm, which is one reason that makes this case considerably more difficult. Another one is that it is no longer possible to gain a power of $t_{j+1}$ in the previous lemma. We therefore first focus on the easier case $\alpha<5$. 
	
	To measure the error simultaneously in two different norms, we define 
	\begin{equation}\label{DefNormSubkrit}
	\|E_n\|_{\tau,j} \coloneqq \max\Big\{\tau^{-1}\|E_n\|_{\ell^\infty_{\tau,j}(L^2 \times \dot H^{-1})}, \tau^{{\gamma_\alpha}-1}\|e_n\|_{\ell^4_{\tau,j}L^{3(\alpha-1)}} \Big\},
	\end{equation}
	where $j \in \N_0$ is a number with $j\tau \le T$ and the parameter ${\gamma_\alpha}$ is given by
	\begin{equation*}
	{\gamma_\alpha} \coloneqq \frac{5}{4}-\frac{1}{\alpha-1} \in \Big[\frac{3}{4},1\Big].
	\end{equation*}
	Hence, $(4,3(\alpha-1),\gamma_\alpha)$ is wave admissible and ${\gamma_\alpha}<1$ if $\alpha<5$. We are now ready to perform the main step of the convergence proof in the subcritical case. We proceed by a ``double induction'', similar as in Section 9 of \cite{ORS}.
	
	\begin{Proposition}\label{PropKonvSubkrit}
		Let $\alpha<5$ and $U=(u,\partial_tu)$, $T$, and $M$ be given by Assumption \ref{Ass}. Define $E_n$ by \eqref{DefMainError}. Then there is a number $\tau_0>0$ depending only on $M$ and $T$ such that we have the estimate
		\[\|E_n\|_{\tau,\lfloor\frac{T}{\tau}\rfloor} \lesssim_{M,T,\alpha} 1 \]
		for all $\tau \in (0,\tau_0]$.
	\end{Proposition}
	\begin{proof}
		\emph{Step 1.} Let $C_1 \ge 1$ be the constant from the discrete Strichartz estimates from Section \ref{SecStrich} with respect to the parameters $(4,3(\alpha-1),\gamma_\alpha)$ and $C_2>0$ be the maximum of the constants from Lemma \ref{LemB}, \ref{LemLokF}, and \ref{LemStab}. Since $\alpha <5$, we can choose a time $T_1 \in (0,T]$ such that 
		\begin{equation}\label{DefT1}
		2C_1C_2T_1^{1-\frac{\alpha-1}{4}} \le \frac{1}{2}.
		\end{equation}
		Next, we define $L\coloneqq \lceil \frac{2T}{T_1} \rceil \in \N$ and the maximum step size $\tau_0>0$ by the relations
		\begin{equation}\label{Deftau0}
		\tau_0 \le T_1, \qquad \tau_0^{1-\gamma_\alpha} \le \frac{1}{4C_2(2C_1)^{L+1}},
		\end{equation}
		where we again exploit that $\alpha<5$. For $\tau \in (0,\tau_0]$ we set $N\coloneqq \lfloor T/\tau \rfloor \in \N$, $N_1\coloneqq \lfloor T_1/\tau \rfloor \in \{1,\dots,N\}$, and $N_m \coloneqq mN_1$ for all $m \in \N_0$. Note that these definitions yield $N \le N_L$. Moreover, we define the number $\ell \coloneqq \lfloor N/N_1\rfloor \in \{1,\dots,L\}$. We thus have the decomposition
		\[[0,t_N] = \bigcup_{m=0}^{\ell-1}[t_{N_m},t_{N_{m+1}}] \cup [t_{N_\ell},t_N], \]
		where each subinterval is of length less or equal $T_1$. 
		To measure the error in each of them, we define the error terms $\mathrm{Err}_m$ by $\mathrm{Err}_{-1}\coloneqq0$,
		\[
		\mathrm{Err}_m \coloneqq \|E_{{N_m}+n}\|_{\tau,N_1}, \quad m \in \{0,\dots,\ell-1\}, \qquad
		\mathrm{Err}_\ell \coloneqq \|E_{{N_\ell}+n}\|_{\tau,N-N_\ell}.
		\] \smallskip
		
		\emph{Step 2.} Our next goal is to show the recursion formula
		\begin{equation}\label{ErrRekursion}
		\mathrm{Err}_{m} \le 2C_1\mathrm{Err}_{m-1} + 4C_1C_2, \quad m \in \{0,\dots,\ell\}.
		\end{equation}
		Note that as soon as \eqref{ErrRekursion} is established for all indices in $\{0,\dots,m\}$, one can deduce the absolute bound
		\begin{equation}\label{ErrAbs}
		\mathrm{Err}_m \le 4C_1C_2\sum_{k=0}^m (2C_1)^k = 4C_1C_2\frac{(2C_1)^{m+1}-1}{2C_1-1} \le 4C_2(2C_1)^{L+1}.
		\end{equation}
		
		a) Fix an index $m \in \{0,\dots,\ell\}$. If $m>0$ we assume that \eqref{ErrRekursion} holds for all indices in $\{0,\dots,m-1\}$. We derive \eqref{ErrRekursion} by proving
		\begin{equation}\label{Indj}
		\|E_{N_m+n}\|_{\tau,j} \le 2C_1\mathrm{Err}_{m-1}+4C_1C_2, \quad j \in \{0,\dots,\min\{N_1,N-N_m\}\},
		\end{equation}
		via induction on $j$. 
		
		b) First let $j=0$. If $m=0$, there is nothing to prove since $E_0=0$. If $m>0$, we directly obtain
		\[\|E_{N_m+n}\|_{\tau,0} \le \|E_{N_{m-1}+n}\|_{\tau,N_1} = \mathrm{Err}_{m-1}, \]
		which in particular shows \eqref{Indj} for $j=0$. Next, let \eqref{Indj} be true for some $j \in \{0,\dots,\min\{N_1,N-N_m\}-1\}$. As in \eqref{ErrAbs} the induction assumption \eqref{ErrRekursion} then yields
		\begin{equation}\label{ErrAbs1}
		\|E_{N_m+n}\|_{\tau,j} \le 4C_2(2C_1)^{L+1}.
		\end{equation}
		We apply the discrete Strichartz estimates from Theorem \ref{ThmDisStrich} and Corollary \ref{KorHybStrich} to the error formula \eqref{FehlerRek}. Observe that \eqref{FehlerRek} implies $E_n = \Pi_{\tau^{-1}}E_n$ inductively. 
		We also use Bernstein's inequality to convert the derivative loss $|\nabla|^{\gamma_\alpha}$ into a loss of $\tau^{-{\gamma_\alpha}}$ in this Strichartz estimate, as explained in Remark \ref{RemDisStr} a). Combined with the definition of the $\|\cdot\|_{\tau,j}$-norm \eqref{DefNormSubkrit}, it follows
		\begin{align*}
		\|E_{N_m+n}\|_{\tau,j+1} &\le \|\Pi_{\tau^{-1}}e^{n\tau A}E_{N_m}\|_{\tau,j+1} + \Big\|\int_0^{t_n}\!\Pi_{\tau^{-1}} e^{(n\tau-s)A}B_{N_m}(s) \dd s\Big\|_{\tau,j+1}\\
		&\quad +  \Big\|\sum_{k=0}^{n-1}\Pi_{\tau^{-1}} e^{(n-k)\tau A}\Big(\Delta_{{N_m}+k} + \tau Q_{{N_m}+k}\Big)\Big\|_{\tau,j+1}\\
		&\le C_1\tau^{-1} \Big( \|E_{N_m}\|_{L^2\times\dot H^{-1}} + \|B_{N_m}\|_{L^1_{t_{j+1}}(L^2 \times \dot H^{-1})} \\
		&\quad + \tau^{-1}\|\Delta_{N_m+n}\|_{\ell^1_{\tau,j}(L^2 \times \dot H^{-1})}+ \|Q_{N_m+n}\|_{\ell^1_{\tau,j}(L^2 \times \dot H^{-1})} \Big)\\
		&\le C_1 \Big( \mathrm{Err}_{m-1} + \tau^{-1}\|g(u)-g(\pi_{\tau^{-1}} u)\|_{L^1_T\dot H^{-1}} \\
		&\quad+ \tau^{-2} \|\Delta_n\|_{\ell^1_{\tau,N}(L^2 \times \dot H^{-1})} \\
		&\quad + \tau^{-1}\|g(\pi_{\tau^{-1}} u(t_{N_m+n}))-g(u_{N_m+n})\|_{\ell^1_{\tau,j} \dot H^{-1}} \Big)\\
		&\le C_1 \mathrm{Err}_{m-1} +2C_1C_2 \\
		&\quad +\tau^{-1}C_1C_2t_{j+1}^{1-\frac{\alpha-1}{4}}\Big(1+\|e_{N_m+n}\|^{\alpha-1}_{\ell^4_{\tau,j}L^{3(\alpha-1)}}\Big)\|e_{N_m+n}\|_{\ell^\infty_{\tau,j}L^2} \\
		&\le C_1 \mathrm{Err}_{m-1} +2C_1C_2 \\
		&\quad +C_1C_2T_{1}^{1-\frac{\alpha-1}{4}}\!\Big(1+\tau^{(1-{\gamma_\alpha})(\alpha-1)}\|E_{N_m+n}\|^{\alpha-1}_{\tau,j}\Big)\|E_{N_m+n}\|_{\tau,j},
		\end{align*}
		where we applied Lemma \ref{LemB}, \ref{LemLokF}, and \ref{LemStab} to bound the error terms. 
		We insert \eqref{ErrAbs1}, the step size restriction $\tau \le \tau_0$ from \eqref{Deftau0} and the definition of $T_1$ from \eqref{DefT1} to obtain
		\begin{align*}
		\|E_{N_m+n}\|_{\tau,j+1} &\le C_1 \mathrm{Err}_{m-1} +2C_1C_2 +2C_1C_2T_{1}^{1-\frac{\alpha-1}{4}}\|E_{N_m+n}\|_{\tau,j}\\
		&\le  C_1 \mathrm{Err}_{m-1} +2C_1C_2 +\frac{1}{2}\|E_{N_m+n}\|_{\tau,j}.
		\end{align*}
		In particular, $\|E_{N_m+n}\|_{\tau,j+1}$ is finite. Since $ \|E_{N_m+n}\|_{\tau,j}\le \|E_{N_m+n}\|_{\tau,j+1}$, it follows that
		\begin{equation*}
		\|E_{N_m+n}\|_{\tau,j+1} \le  2C_1 \mathrm{Err}_{m-1} +4C_1C_2,
		\end{equation*}
		which closes the induction on $j$. Hence, the recursion \eqref{ErrRekursion} is true. The assertion now follows from \eqref{ErrAbs}.
	\end{proof}
	
	\begin{proof}[Proof of Theorem \ref{Thm1}]
		We take $\tau_0>0$ from Proposition \ref{PropKonvSubkrit} and infer
		\begin{align*}
		\|&U(t_n)-U_n\|_{L^2 \times \dot H^{-1}}\\
		&\le \|U(t_n)-\Pi_{\tau^{-1}}U(t_n)\|_{L^2 \times \dot H^{-1}} + \|\Pi_{\tau^{-1}}U(t_n)-U_n\|_{L^2 \times \dot H^{-1}}\\ 
		&\le \tau\|U(t_n)\|_{\dot H^1 \times L^2} + \|E_n\|_{L^2 \times \dot H^{-1}} \lesssim_{M,T,\alpha} \tau,
		\end{align*}
		using the estimates from Lemma \ref{LemPi} and Proposition \ref{PropKonvSubkrit}.
	\end{proof}
	
	\section{The critical case}\label{SecKrit}
	
	We now consider the case $\alpha=5$. As a first auxiliary result we establish convergence of the time-discrete Strichartz norm towards the time-continuous Strichartz norm. This will be done for the homogeneous part of the evolution, i.e.,
	\[S(t)(f,v) = \cos(t|\nabla|)f + |\nabla|^{-1} \sin(t|\nabla|)v,\]
	on a bounded interval $J$. We recall the notation
	\[\|F_n\|_{\ell^p_\tau(J,X)} = \Big(\tau \sum_{\substack{n \in \Z\\ n\tau \in J}} \|F_n\|_{X}^p \Big)^{\frac{1}{p}}.  \]
	
	\begin{Lemma}\label{LemKonvDiskrStet}
		Let $(p,q,\gamma)$ be wave admissible, $f \in \dot H^\gamma$, $v \in \dot H^{\gamma-1}$, and $J \subseteq \R$ be a bounded interval. Then we have the convergence
		\begin{align*}
		\|e^{\iu n \tau |\nabla|}\pi_{\tau^{-1}} f\|_{\ell^p_\tau(J,L^q)} &\to \|e^{\iu t |\nabla|}f\|_{L^p_JL^q},\\
		\|S(n\tau)\Pi_{\tau^{-1}} (f,v)\|_{\ell^p_\tau(J,L^q)} &\to \|S(t)(f,v)\|_{L^p_JL^q},
		\end{align*}
		as $\tau \to 0$. 
	\end{Lemma}
	\begin{proof}
		Let $\eps>0$. We choose a Schwartz function $\varphi \in \mathcal S$ such that 
		$\supp \hat \varphi$ is compact
		and $\|f-\varphi\|_{\dot H^\gamma} \le \eps$. The function $t \mapsto e^{\iu t |\nabla|}\varphi$ is then continuous with values in $L^q$, since Hausdorff--Young and dominated convergence yield
		\[\|e^{\iu s |\nabla|}\varphi-e^{\iu t|\nabla|}\varphi\|_{L^q} = \|\mathcal F^{-1}((e^{\iu s |\xi|}-e^{\iu t |\xi|})\hat \varphi)\|_{L^q} \le \|(e^{\iu s |\xi|}-e^{\iu t |\xi|})\hat \varphi\|_{L^{q'}} \to 0 \]
		as $s \to t$. We also have $\pi_{\tau^{-1}} \varphi = \varphi$ for $\tau$ small enough, because $\supp \hat \varphi$ is compact. It follows
		\[\|e^{\iu n \tau |\nabla|}\pi_{\tau^{-1}} \varphi\|_{\ell^p_\tau(J,L^q)} \to \|e^{\iu t |\nabla|}\varphi\|_{L^p_JL^q} \] 
		as $\tau \to 0$, as there are essentially Riemann sums on the left-hand side. Putting things together, we conclude
		\begin{align*}
		\Big|\|&e^{\iu n \tau |\nabla|}\pi_{\tau^{-1}} f\|_{\ell^p_\tau(J,L^q)} - \|e^{\iu t |\nabla|}f\|_{L^p_JL^q}\Big|\\ &\le \|e^{\iu n \tau |\nabla|}\pi_{\tau^{-1}} (f-\varphi)\|_{\ell^p_\tau(J,L^q)} + \Big|\|e^{\iu n \tau |\nabla|}\pi_{\tau^{-1}} \varphi\|_{\ell^p_\tau(J,L^q)} - \|e^{\iu t |\nabla|}\varphi\|_{L^p_JL^q}\Big| \\
		&\quad + \|e^{\iu t |\nabla|}(\varphi-f)\|_{L^p_JL^q}\\
		&\lesssim_{p,q} 2 \|f-\varphi\|_{\dot H^\gamma} + \eps \le 3\eps
		\end{align*}
		for $\tau$ small enough, using the reverse triangle inequality and Strichartz estimates from Theorem \ref{ThmStrich} and \ref{ThmDisStrich}. Thus, the first assertion is shown. The proof of the second one follows the same lines.
	\end{proof}

	To show the error bound in the critical case, we use a regularization argument. For $H^2 \times H^1$ initial data, it requires first-order convergence of the scheme \eqref{LieFilt} in the $\dot H^1 \times L^2$-norm. To use this fact, we also need the continuous dependence on the initial data, both for the equation \eqref{NLW} and the scheme \eqref{LieFilt}. We show these results under a smallness condition on a Strichartz norm of the orbit, which can always be fulfilled by choosing a small end time $b$, see Theorem \ref{ThmStrich}.
	We note that $b$ in Proposition \ref{PropLokStab} corresponds to a possible choice of $T_0$ in Theorem \ref{ThmLokWoh}, cf.\ \cite{Sogge,Tao2006}.
	In the following we frequently use the discrete Duhamel formula
	\begin{equation}\label{DiskrDuhamel}
	U_{n} = e^{n\tau A}\Pi_{\tau^{-1}} U^0 + \tau \sum_{k=0}^{n-1}\Pi_{\tau^{-1}} e^{(n-k)\tau A} G(U_{k})
	\end{equation}
	for the scheme \eqref{LieFilt}, see \eqref{DiskrDuhamelMitM}. The first component of \eqref{DiskrDuhamel} reads as
	\begin{equation}\label{DiskrDuh1}
	u_n = S(t_n)\Pi_{\tau^{-1}} U^0 + \tau \sum_{k=0}^{n-1} |\nabla|^{-1} \sin(t_{n-k}|\nabla|)\pi_{\tau^{-1}} g(u_k), 
	\end{equation}
	similar as in the continuous case \eqref{Duh1}.
	
	\begin{Proposition}\label{PropLokStab}
		Let $R>0$. Then there is a radius $\delta_0 = \delta_0(R)>0$ such that for any $\delta \in (0,\delta_0]$ the following is true. For all $W^0 \in \dot H^1 \times L^2$ with $\|W^0\|_{\dot H^1\times L^2} \le R$ and every $b>0$ with $\|S(\cdot)W^0\|_{L^4_bL^{12}} \le \delta$, there is a time step $\bar \tau = \bar \tau(\delta,W^0,b)>0$ such that the next assertions hold. \smallskip
		
		a) For every $Y^0, Z^0 \in \overline B_{\dot H^1 \times L^2}(W^0,\delta)$, the solutions $Y$ and $Z$ of \eqref{NLW} with $\alpha=5$ and initial values $Y^0$ resp.\ $Z^0$ exist on $[0,b]$. Moreover, we then have the estimates
		\begin{align}
		\label{yklein}\|\pi_{\tau^{-1}} y(t_n)\|_{\ell^4_\tau([0,b],L^{12})} &\le \kappa\delta, \\
		\label{ynklein}\|y_n\|_{\ell^4_\tau([0,b],L^{12})} &\le \kappa\delta,\\	
		\label{StetAbh}\|Y-Z\|_{L^\infty([0,b],\dot H^1 \times L^2)} &\le 2\|Y^0-Z^0\|_{\dot H^1 \times L^2}, \\
		\label{LokStab}\|Y_n-Z_n\|_{\ell^\infty_\tau([0,b],\dot H^1\times L^2)} &\le 2\|Y^0-Z^0\|_{\dot H^1 \times L^2},
		\end{align}
		for all $\tau \in (0,\bar \tau]$, where $Y_n$ (resp.\ $Z_n$) are the iterates of \eqref{LieFilt} for initial values $\Pi_{\tau^{-1}} Y^0$ (resp.\ $\Pi_{\tau^{-1}} Z^0$), $y_n$ (resp.\ $y$) is the first component of $Y_n$ (resp.\ $Y$), and $\kappa>0$ is a constant. \smallskip
		
		b) If $Y^0 \in H^2 \times H^1$ satisfies $\|Y^0-W^0\|_{\dot H^1 \times L^2} \le \delta/2$, then there is a constant $C = C(\|Y^0\|_{H^2\times H^1})>0$ such that the error bound
		\begin{equation}\label{KonvBetterData}
		\|Y(t_n)-Y_n\|_{\ell^\infty_\tau([0,b],\dot H^1 \times L^2)} \le C\tau
		\end{equation}
		holds for all $\tau \in (0,\bar \tau]$.
	\end{Proposition}
	\begin{proof}
		\emph{Step 1.} Let $C_\mathrm{So} \ge 1$ be the constant from the Sobolev embedding $\dot H^1 \hookrightarrow L^6$ and $C_1 \ge 1$ be the constant from the Strichartz estimates in Section \ref{SecStrich}, where we choose the exponents $(p,q,1) = (4,12,1)$. We define
		\begin{equation}\label{DefDeltaInLokStab}
		\delta_0 \coloneqq \min\Big\{R, (3C_\mathrm{So}C_1(3+C_1)^4R)^{-\frac13} 
		,(10C_\mathrm{So}C_1(3+C_1)^4)^{-\frac14} \Big\}.
		\end{equation}
		Let $\delta \in (0,\delta_0]$. Since by assumption $\|S(\cdot)W^0\|_{L^4_bL^{12}} \le \delta$, Lemma \ref{LemKonvDiskrStet} yields a stepsize $\bar \tau>0$ such that
		\begin{equation}\label{LokStabKleinheit}
		\|S(t_n)\Pi_{\tau^{-1}} W^0\|_{\ell^4_\tau([0,b],L^{12})} \le 2\delta
		\end{equation}
		for all $\tau \in (0,\bar \tau]$.
		We first show 
		\begin{equation}\label{ClaimInLokStab}
		\|Y_n\|_{\ell^\infty_{\tau,j}(\dot H^1 \times L^2)} \le 3R, \qquad \|y_n\|_{\ell^4_{\tau,j}L^{12}} \le (3+C_1)\delta
		\end{equation}
		for all $\tau \in (0,\bar \tau]$ and $j \in \N_0$ with $j\tau \le b$. This in particular shows the inequality \eqref{ynklein} with $\kappa \coloneqq 3+C_1$. 
		
		We proceed by induction on $j$. For $j=0$, we clearly have 
		\begin{equation}\label{LokStabIndj=0}
		\|Y_0\|_{\dot H^1 \times L^2} \le \|Y^0-W^0\|_{\dot H^1 \times L^2} + \|W^0\|_{\dot H^1 \times L^2} \le \delta + R \le 2R 
		\end{equation}
		since $\delta \le R$. Theorem \ref{ThmDisStrich} and \eqref{LokStabKleinheit} further imply
		\begin{align*}
		\|y_n\|_{\ell^4_{\tau,0}L^{12}} &= \tau^{\frac14} \|S(0)\Pi_{\tau^{-1}} Y^0\|_{L^{12}} \\
		&\le \tau^{\frac14} \|S(0)\Pi_{\tau^{-1}} (Y^0-W^0)\|_{L^{12}} + \tau^{\frac14} \|S(0)\Pi_{\tau^{-1}} W^0\|_{L^{12}}\\
		&\le C_1\|Y^0-W^0\|_{\dot H^1 \times L^2} + 2\delta \le C_1\delta + 2\delta = (2+C_1)\delta.
		\end{align*}
		
		For the induction step $j \leadsto j+1$, we assume that \eqref{ClaimInLokStab} holds for some $j \in \N_0$ with $(j+1)\tau \le b$. We compute
		\begin{align*}
		\|&Y_n\|_{\ell^\infty_{\tau,j+1}(\dot H^1 \times L^2)} \\
		&\le \|e^{n\tau A}\Pi_{\tau^{-1}} Y^0\|_{\ell^\infty_{\tau,j+1}(\dot H^1 \times L^2)} + \tau \Big\|\sum_{k=0}^{n-1}\Pi_{\tau^{-1}} e^{(n-k)\tau A} G(Y_{k})\Big\|_{\ell^\infty_{\tau,j+1}(\dot H^1 \times L^2)}\\
		&\le \|Y^0\|_{\dot H^1\times L^2} + \||y_n|^4y_n\|_{\ell^1_{\tau,j}L^2} \le 2R + \|y_n\|^4_{\ell^4_{\tau,j}L^{12}}\|y_n\|_{\ell^\infty_{\tau,j}L^{6}}\\
		&\le 2R + 3C_\mathrm{So}(3+C_1)^4\delta^4R \le 3R,
		\end{align*}
		by means of the discrete Duhamel formula \eqref{DiskrDuhamel}, \eqref{LokStabIndj=0}, Hölder's inequality, the induction assumption \eqref{ClaimInLokStab}, and the definition of $\delta$ from \eqref{DefDeltaInLokStab}. Similarly, using \eqref{DiskrDuh1} and Theorem \ref{ThmDisStrich}, we estimate
		\begin{align*}
		\|y_n\|_{\ell^4_{\tau,j+1}L^{12}}
		&\le \|S(t_n)\Pi_{\tau^{-1}} Y^0\|_{\ell^4_{\tau,j+1}L^{12}} \\
		&\quad + \tau \Big\|\sum_{k=0}^{n-1} |\nabla|^{-1} \sin(t_{n-k}|\nabla|)\pi_{\tau^{-1}} g(y_k)\Big\|_{\ell^4_{\tau,j+1}L^{12}}\\
		&\le \|S(t_n)\Pi_{\tau^{-1}} (Y^0-W^0)\|_{\ell^4_{\tau,j+1}L^{12}} + \|S(t_n)\Pi_{\tau^{-1}} W^0\|_{\ell^4_{\tau,j+1}L^{12}} \\
		&\quad + C_1\||y_n|^4y_n\|_{\ell^1_{\tau,j}L^2}\\
		&\le C_1\delta + 2\delta + 3C_\mathrm{So}C_1(3+C_1)^4\delta^4R \le (3+C_1)\delta.
		\end{align*}
		Hence, the claim \eqref{ClaimInLokStab} is true for all $j\tau \le b$. \smallskip
		
		\emph{Step 2.} Estimate \eqref{LokStab} is shown by an analogous argument starting from \eqref{DiskrDuhamel}. Using also \eqref{ynklein} for $z_n$, we deduce the inequality 
		\begin{align*}
		\|&Y_n-Z_n\|_{\ell^\infty_\tau([0,b],\dot H^1\times L^2)}\\
		&\le \|Y^0-Z^0\|_{\dot H^1 \times L^2} + \||y_n|^4y_n - |z_n|^4z_n\|_{\ell^1_\tau([0,b],L^2)} \\
		&\le \|Y^0-Z^0\|_{\dot H^1 \times L^2} + \frac52\|(|y_n|^4+|z_n|^4)|y_n - z_n|\|_{\ell^1_\tau([0,b],L^2)}\\
		&\le \|Y^0-Z^0\|_{\dot H^1 \times L^2}\\
		&\quad + \frac52\Big(\|y_n\|^4_{\ell^4_\tau([0,b],L^{12})}+\|z_n\|^4_{\ell^4_\tau([0,b],L^{12})}\Big) \|y_n - z_n\|_{\ell^\infty_\tau([0,b],L^6)} \\
		&\le \|Y^0-Z^0\|_{\dot H^1 \times L^2} + 5C_\mathrm{So}(3+C_1)^4\delta^4 \|Y_n-Z_n\|_{\ell^\infty_\tau([0,b],\dot H^1\times L^2)} \\
		&\le \|Y^0-Z^0\|_{\dot H^1 \times L^2} + \frac12 \|Y_n-Z_n\|_{\ell^\infty_\tau([0,b],\dot H^1\times L^2)}, 
		\end{align*}
		which in turn implies 
		\[\|Y_n-Z_n\|_{\ell^\infty_\tau([0,b],\dot H^1\times L^2)} \le 2\|Y^0-Z^0\|_{\dot H^1 \times L^2}, \]
		as desired. \smallskip
		
		\emph{Step 3.} 
		The existence of the continuous solutions $Y$ and $Z$ until time $b$ as well as the estimate \eqref{StetAbh} are part of the known local wellposedness theory of \eqref{NLW}, cf.\ Theorem \ref{ThmLokWoh} and Chapter 5.1 of \cite{Tao2006}. We therefore omit the proof. To carry it out, one can proceed analogously to Step 1 and 2, replacing the discrete norms by the continuous ones and the induction by a fixed point argument. The estimate \eqref{yklein} in discrete Strichartz norm can be shown in a similar way as \eqref{ynklein}. \smallskip
		
		\emph{Step 4.} Now we want to show the error bound \eqref{KonvBetterData} for better data $Y^0 \in H^2 \times H^1$ with $\|Y^0-W^0\|_{\dot H^1 \times L^2} \le \delta/2$. Since the nonlinearity $G$ leaves the space $H^2 \times H^1$ invariant and is Lipschitz continuous on balls, there is a unique solution $\tilde Y \in C([0,T_{\mathrm{max}}), H^2 \times H^1)$ of \eqref{NLW} with initial data $Y^0$ on a maximal existence interval $[0,T_\mathrm{max})$. By Sobolev's embedding, the integrability condition $\tilde y \in L^4_{\mathrm{loc}}([0,T_{\mathrm{max}}),L^{12})$ is satisfied. Hence, $\tilde Y$ coincides with the $\dot H^1 \times L^2$-solution $Y$ on $[0,b]$ by uniqueness, as long as they are both defined. 
		
		In the following, we show that $T_\mathrm{max} >b$. By a standard blow-up criterion, it suffices to show that $\|Y\|_{L^\infty([0,b],H^2 \times H^1)}$ is finite. 
		First, note that $(y,\partial_ty)$ belongs to $C([0,b], \dot H^1 \times L^2)$ and that the $L^2$-norm of $y$ stays bounded since
		\begin{align*}
		\|y(t)\|_{L^2} &\le \|y(0)\|_{L^2} + \int_0^t \|\partial_ty(s)\|_{L^2} \dd s \le \|Y^0\|_{H^2 \times H^1} + b\|\partial_ty\|_{L^\infty([0,b],L^2)}\\
		&<\infty 
		\end{align*}
		for all $t \in [0,b]$. 
		
		For the boundedness in the $\dot H^2 \times \dot H^1$-norm, we use that the Sobolev norm of a function can be expressed by bounds on the norms of difference quotients.
		For any $h\in \R^3$, we introduce the spatial translation operator $\mathcal T_{h}$ by $(\mathcal T_{h}(f,g))(x) \coloneqq (f(x+h),g(x+h))$, where $f$ and $g$ are functions on $\R^3$. By Proposition 9.3 of \cite{Brezis},
		\[\| {\mathcal T_{h}}Y^0-Y^0\|_{\dot H^1 \times L^2} \lesssim |h|\|Y^0\|_{\dot H^2\times \dot H^1}. \]
		Therefore, there is a number $h_0>0$ with $\| {\mathcal T_{h}}Y^0-Y^0\|_{\dot H^1 \times L^2} \le \delta/2$ for all $|h| \le h_0$. From now on we assume that $|h| \le h_0$. The triangle inequality yields $\| {\mathcal T_{h}}Y^0-W^0\|_{\dot H^1 \times L^2} \le \delta$. Since $\mathcal T_{h}Y$ solves \eqref{NLW2} with initial value $\mathcal T_{h}Y^0$, from \eqref{StetAbh} we can deduce that
		\begin{equation*}
		\|\mathcal T_{h}Y-Y\|_{L^\infty([0,b],\dot H^1 \times L^2)} \le 2\|\mathcal T_{h}Y^0-Y^0\|_{\dot H^1 \times L^2} \lesssim |h|\|Y^0\|_{\dot H^2\times \dot H^1}.
		\end{equation*}
		Proposition 9.3 of \cite{Brezis} now yields $Y(t) \in H^2 \times H^1$ for $t \in [0,b]$ and
		\begin{equation}\label{H2Bound}
		\|Y\|_{L^\infty([0,b],\dot H^2 \times \dot H^1)} \lesssim \|Y^0\|_{\dot H^2\times \dot H^1}.
		\end{equation}
		Thus, $T_\mathrm{max} >b$. \smallskip
		
		\emph{Step 5.} We can now estimate the error $\|Y(t_n)-Y_n\|_{\ell^\infty_\tau([0,b],\dot H^1 \times L^2)}$. Let $\tau \in (0,\bar\tau]$. For the rest of this proof, we allow our implicit constants to depend on $\|Y^0\|_{H^2 \times H^1}$. First, Lemma \ref{LemPi} and \eqref{H2Bound} imply
		\begin{align}\label{BessFehProj}
		\|&Y(t_n)-Y_n\|_{\ell^\infty_\tau([0,b],\dot H^1 \times L^2)}\nonumber \\
		&\le \|(I-\Pi_{\tau^{-1}})Y(t_n)\|_{\ell^\infty_\tau([0,b],\dot H^1 \times L^2)} + \|\Pi_{\tau^{-1}} Y(t_n)-Y_n\|_{\ell^\infty_\tau([0,b],\dot H^1 \times L^2)} \nonumber\\
		&\le \tau \|Y(t_n)\|_{\ell^\infty_\tau([0,b],\dot H^2 \times \dot H^1)} + \|\Pi_{\tau^{-1}} Y(t_n)-Y_n\|_{\ell^\infty_\tau([0,b],\dot H^1 \times L^2)} \nonumber\\
		&\lesssim \tau  + \|\Pi_{\tau^{-1}} Y(t_n)-Y_n\|_{\ell^\infty_\tau([0,b],\dot H^1 \times L^2)}
		. 
		\end{align}
		To estimate $\Pi_{\tau^{-1}} Y(t_n)-Y_n$, we use the expressions from Proposition \ref{PropFehlerRek} to write
		\begin{align*}
		\Pi_{\tau^{-1}} &Y(t_n)-Y_n\\
		&= \int_0^{t_n}\Pi_{\tau^{-1}} e^{(n\tau-s)A}B_0(s) \dd s +  \sum_{k=0}^{n-1}\Pi_{\tau^{-1}} e^{(n-k)\tau A}\Big(\Delta_{k} + \tau Q_{k}\Big),
		\end{align*}
		where the terms $B_0$, $\Delta_n$ and $Q_n$ now include $Y$ instead of $U$. A direct estimate gives 
		\begin{align}\label{BessFeh}
		 \|&\Pi_{\tau^{-1}}Y(t_n)-Y_n\|_{\ell^\infty_\tau([0,b],\dot H^1 \times L^2)} \nonumber\\
		& \le \||y|^4y-|\pi_{\tau^{-1}} y|^4\pi_{\tau^{-1}} y\|_{L^1_bL^2}\nonumber\\
		&\quad +  \Big\|\frac 1\tau\int_0^\tau e^{-sA} G(\Pi_{\tau^{-1}} Y(t_n+s)) \dd s - G(\Pi_{\tau^{-1}} Y(t_n))\Big\|_{\ell^1_\tau([0,b-\tau],\dot H^1 \times L^2)}\nonumber\\
		&\quad + \||\pi_{\tau^{-1}} y(t_n)|^4\pi_{\tau^{-1}} y(t_n)-|y_n|^4y_n\|_{\ell^1_\tau([0,b],L^2)}.
		\end{align}
		Here we interpret $[0,b-\tau] = \emptyset$ if $\tau>b$. The three summands in \eqref{BessFeh} can be bounded similarly as in the Lemmas \ref{LemB}, \ref{LemLokF}, and \ref{LemStab}. For the first summand, using Lemma \ref{LemPi} we infer 
		\begin{align*}
		\||y|^4y-|\pi_{\tau^{-1}} y|^4\pi_{\tau^{-1}} y\|_{L^1_bL^2} &\lesssim \|(|y|^4+|\pi_{\tau^{-1}} y|^4)|y-\pi_{\tau^{-1}} y|\|_{L^1_bL^2}\\
		&\le \Big(\|y\|^4_{L^4_bL^{12}} \!+ \|\pi_{\tau^{-1}} y\|^4_{L^4_bL^{12}}\!\Big)\|(I\!-\pi_{\tau^{-1}})y\|_{L^\infty_bL^6} \\
		&\lesssim \|(I-\pi_{\tau^{-1}})y\|_{L^\infty_b \dot H^1} \lesssim \tau \|y\|_{L^\infty_b \dot H^2} \lesssim \tau.
		\end{align*}
		Here, the bounds for the $L^4_bL^{12}$-norm follow from Theorem \ref{ThmLokWoh} and Remark \ref{RemStrichU}, since our $b$ corresponds to a possible choice of $T_0$ in Theorem \ref{ThmLokWoh}. By \eqref{ErrRepres} and Proposition \ref{PropDiskrStrichU}, the second summand is bounded by
		\begin{align*}
		&\frac1\tau\Big\|\int_0^\tau \int_0^s e^{-\sigma A}\begin{pmatrix}-g(\pi_{\tau^{-1}} y(t_n+\sigma)) \\ g'(\pi_{\tau^{-1}} y(t_n+\sigma))\pi_{\tau^{-1}} \partial_t y(t_n+\sigma)\end{pmatrix} \dd \sigma\dd s \Big\|_{\ell^1_\tau([0,b-\tau],\dot H^1 \times L^2)} \\
		&\lesssim \tau \sup_{\sigma \in [0,\tau]}\|\pi_{\tau^{-1}}y(t_n+\sigma)\|^4_{\ell^4_\tau([0,b-\tau], L^{12})} \Big( \||\nabla\pi_{\tau^{-1}} y|\|_{L^\infty_bL^6} + \|\pi_{\tau^{-1}} \partial_t y\|_{L^\infty_b L^6} \Big) \\
		&\lesssim \tau.
		\end{align*}
		Finally, we estimate the last part of \eqref{BessFeh} by
		\begin{align*}
		\||&\pi_{\tau^{-1}} y(t_n)|^4\pi_{\tau^{-1}} y(t_n)-|y_n|^4y_n\|_{\ell^1_\tau([0,b],L^2)} \\
		&\le \frac52 \Big(\|\pi_{\tau^{-1}} y(t_n)\|^4_{\ell^4_\tau([0,b],L^{12})} + \|y_n\|^4_{\ell^4_\tau([0,b],L^{12})}  \Big)\|\pi_{\tau^{-1}} y(t_n)-y_n\|_{\ell^\infty_\tau([0,b],L^6)} \\
		&\le 5 C_\mathrm{So}(3+C_1)^4 \delta^4 \|\Pi_{\tau^{-1}} Y(t_n)-Y_n\|_{\ell^\infty_\tau([0,b],\dot H^1 \times L^2)} \\
		&\le \frac12 \|\Pi_{\tau^{-1}} Y(t_n)-Y_n\|_{\ell^\infty_\tau([0,b],\dot H^1 \times L^2)},
		\end{align*}
		using \eqref{yklein}, \eqref{ynklein} and the definition of $\delta_0$ in \eqref{DefDeltaInLokStab}. This term can be absorbed by the left-hand side of \eqref{BessFeh}. Putting things together, \eqref{BessFehProj} and \eqref{BessFeh} imply \eqref{KonvBetterData}.
	\end{proof}
	
	Proposition \ref{PropLokStab} only gives a local statement on a possibly small time interval $[0,b]$. Since we want to show a global error bound on the potentially much larger interval $[0,T]$, we need to apply Proposition \ref{PropLokStab} recursively. To this aim, we first have to iterate the smallness condition in $L^4_bL^{12}$.

	\begin{Lemma}\label{LemUklein}
		Let $U=(u,\partial_tu)$, $T$, and $M$ be given by Assumption \ref{Ass} with $\alpha=5$ and let $\delta>0$. Then there are a number $L \in \N$ and times $0=T_0 < T_1 < \dots < T_L = T$, such that the inequality
		\[
		\|S(\cdot)U(T_m)\|_{L^4_{b_m}L^{12}} \le \delta
		\]
		holds for all $m \in \{0,\dots,L-1\}$, where we set $b_m \coloneqq T_{m+1}-T_m >0$. The number $L \in \N$ only depends on $\delta$, $M$, and $T$.
	\end{Lemma}
	\begin{proof}
		Let $C$ be the constant of the Strichartz estimates from Theorem \ref{ThmStrich} with respect to the exponents $(4,12,1)$. We define
		\begin{equation}\label{Defr}
		r \coloneqq \min\Big\{\frac \delta 2, \Big(\frac{\delta}{2CM}\Big)^\frac14\Big\}.
		\end{equation}
		Since $\|u\|_{L^4_TL^{12}} \le M$ is finite, we can find times $0=T_0 < T_1 < \dots < T_L = T$, such that the inequality
		\[
		\|u\|_{L^4([T_m,T_{m+1}],L^{12})} \le r
		\]
		holds for all $m \in \{0,\dots,L-1\}$. Here we can choose $L = \lceil\|u\|^4_{L^4_TL^{12}}/r^4 \rceil.$
		Let now $m \in \{0,\dots,L-1\}$ and $b_m \coloneqq T_{m+1}-T_m >0$. Starting from \eqref{Duh1} and \eqref{Duh1lin}, Theorem \ref{ThmStrich} and \eqref{Defr} imply
		\begin{align*}	
		&\|S(\cdot)U(T_m)\|_{L^4_{b_m}L^{12}} \\
		&\le \|u(T_m+\cdot)\|_{L^4_{b_m}L^{12}} +  \Big\|\int_0^t |\nabla|^{-1} \sin((t-s)|\nabla|)g(u(T_m+s)) \dd s \Big\|_{L^4_{b_m}L^{12}} \\
		&\le r + C\||u|^4u\|_{L^1([T_m,T_{m+1}],L^2)} \le r + CM\|u\|^4_{L^4([T_m,T_{m+1}],L^{12})}\\
		&\le r + CMr^4 \le \delta. \qedhere
		\end{align*}
	\end{proof}
	
	We now show the global error bound for the critical case. We use ideas from the proof of Theorem 1.6 in \cite{ChoiKoh}, where similar arguments were used in the context of nonlinear Schrödinger equations, but only in the energy-subcritical case.
	The proof will be divided in three steps. In the first step, we define the needed variables and divide the interval $[0,T]$ into a finite number of subintervals, which are so small that we can apply Proposition \ref{PropLokStab} on each of them. In the second step, we first prove the convergence of the scheme in the $\dot H^1 \times L^2$-norm without any convergence rate. This fact ensures that the discrete approximation stays close to the solution in the $\dot H^1 \times L^2$-norm if $\tau$ is small enough. We can then apply Proposition \ref{PropLokStab} iteratively. Finally, in the last step, we estimate the error in the $L^2 \times \dot H^{-1}$-norm to obtain the convergence of order one.
	In contrast to Theorem \ref{Thm1}, the maximum step size $\tau_0$ will now not only depend on the size $M$ of the solution $u$ and on the end time $T$, but also on further properties of the solution $u$ to \eqref{NLW2}.

	\begin{proof}[Proof of Theorem \ref{ThmKrit}]
		\emph{Step 1.} Let $R\coloneqq \|U\|_{L^\infty([0,T],\dot H^1 \times L^2)} \le M < \infty$. We take $\delta_0 = \delta_0(R)$ given by Proposition \ref{PropLokStab}. 
		We define
		\begin{equation}\label{Defdelta}
		\delta \coloneqq \min\Big\{\delta_0,\frac{1}{\kappa(10C_\mathrm{So})^{\frac14}}\Big\},
		\end{equation}
		where $\kappa>0$ is the constant from Proposition \ref{PropLokStab} and  $C_\mathrm{So} >0$ is the norm of the Sobolev embedding $L^{6/5} \hookrightarrow \dot H^{-1}$.
		Lemma \ref{LemUklein} provides a number $L \in \N$ and times $0=T_0 < T_1 < \dots < T_L = T$ such that
		\begin{equation}\label{KritUKlein}
		\|S(\cdot)U(T_m)\|_{L^4_{b_m}L^{12}} \le \delta
		\end{equation}
		holds for all $m \in \{0,\dots,L-1\}$, where $b_m \coloneqq T_{m+1}-T_m >0$. Here, the number $L \in \N$ only depends on $M$ and $T$.
		We now define 
		\begin{equation}\label{Eps0}
		\eps \coloneqq \frac{\delta}{9\cdot 2^{L}}.
		\end{equation}
		By continuity of $U$, there is a number $\rho>0$ such that
		\begin{equation}\label{KritGlmStet}
		\|U(T_m)-U(t)\|_{\dot H^1\times L^2} \le \eps
		\end{equation}
		for all $m \in \{1,\dots,L\}$ and $t \in [0,T]$ with $|T_m-t|\le \rho$. We pick functions $Y^0, \dots, Y^L \in H^2\times H^1$ with
		\begin{equation}\label{Ys}
		\|Y^m-U(T_m)\|_{\dot H^1 \times L^2} \le \eps\le \frac\delta 2
		\end{equation} 
		for all $m \in \{0,\dots,L\}$.
		Due to Theorem \ref{ThmStrich}, we find a time $b_L>0$ such that 
		\begin{equation}\label{SUklein}
		\|S(\cdot)U(T)\|_{L^4_{b_L}L^{12}} \le \delta. 
		\end{equation}
		We define the maximal step size $\tau_0>0$ by
		\begin{equation}\label{KritTau0}
		\tau_0 \coloneqq \min\Big\{
		\frac \rho L,\frac{b_L}{L}, \min_{m=0,\dots,L} \frac{\eps}{C(Y^m)}, \min_{m=0,\dots,L}\bar \tau(\delta,U(T_m),b_m) \Big\} ,
		\end{equation}
		where the numbers $C(Y^m) = C(\|Y^m\|_{H^2\times H^1})$ and $\bar \tau(\delta,U(T_m),b_m)$ are taken from Proposition \ref{PropLokStab}. 
		
		Let $\tau \in (0,\tau_0]$. To decompose the interval, for any $m \in \{0,\dots,L\}$ we set
		\[N_m \coloneqq \sum_{j=0}^{m-1}\Big\lfloor \frac{b_j}{\tau}\Big\rfloor \in \N_0. \]
		The intervals $J_m$ are defined as $J_m\coloneqq[t_{N_m},t_{N_{m+1}}]$ if $m \in \{0,\dots,L-1\}$ and $J_L \coloneqq [t_{N_L},T]$.
		Hence, we have 
		\[[0,T] = \bigcup_{m=0}^{L}J_m.\]
		By construction, each subinterval $J_m$ is of length less or equal $b_m$. This also holds for the last interval $J_L$, because of
		\begin{equation}\label{t-T_m}
		T-t_{N_L} = \sum_{m=0}^{L-1}\Big(b_m-\Big\lfloor \frac{b_m}{\tau}\Big\rfloor\tau\Big) \le \sum_{m=0}^{L-1}\Big(b_m-\Big( \frac{b_m}{\tau}-1\Big)\tau\Big) = L\tau \le b_L, 
		\end{equation}
		where we use \eqref{KritTau0}.
		\smallskip
		
		\emph{Step 2.} We prove the convergence in the $\dot H^1 \times L^2$-norm, namely
		\begin{equation}\label{KonvH1}
		\|U(t_n)-U_n\|_{\ell^\infty_\tau([0,T],\dot H^1 \times L^2)} \to 0
		\end{equation}
		as $\tau \to 0$ (without any rate). To measure the error in each subinterval $J_m$, we define the error norms $\mathrm{Err}_m$ by $\mathrm{Err}_{-1}\coloneqq0$ and
		\begin{equation*}
		\mathrm{Err}_m \coloneqq \|U(t_n)-U_n\|_{\ell^\infty_\tau(J_m,\dot H^1 \times L^2)}, \quad m \in \{0,\dots,L\}.
		\end{equation*}
		Next, we show the recursion formula
		\begin{equation}\label{KritErrRekursion}
		\mathrm{Err}_{m} \le 2\mathrm{Err}_{m-1} + 9\eps, \quad m \in \{0,\dots,L\}
		\end{equation}
		via induction on $m$. First, let $m=0$. We introduce the notation $U(t,W^0) \coloneqq W(t)$, where $W$ is the solution of \eqref{NLW} with initial value $W^0$. Recall the definition of $\Phi_\tau$ in \eqref{LieFilt}. We get
		\begin{align*}
		\mathrm{Err}_0 &= \|U(t_n)-U_n\|_{\ell^\infty_\tau(J_0,\dot H^1 \times L^2)}\\  
		&\le \|U(t_n,U^0)-U(t_n,Y^0)\|_{\ell^\infty_\tau(J_0,\dot H^1 \times L^2)} \\
		&\quad + \|U(t_n,Y^0)-\Phi_\tau^n(\Pi_{\tau^{-1}} Y^0)\|_{\ell^\infty_\tau(J_0,\dot H^1 \times L^2)}\\
		&\quad + \|\Phi_\tau^n(\Pi_{\tau^{-1}} Y^0)-\Phi_\tau^n(\Pi_{\tau^{-1}} U^0)\|_{\ell^\infty_\tau(J_0,\dot H^1 \times L^2)}\\
		&\le 2\|U^0-Y^0\|_{\dot H^1 \times L^2} + C(Y^0)\tau + 2\|U^0-Y^0\|_{\dot H^1 \times L^2} \le 5\eps,
		\end{align*}
		using the estimates from Proposition \ref{PropLokStab} and the relations \eqref{KritUKlein}, \eqref{Ys} and \eqref{KritTau0}. 
		
		For the induction step $m-1 \leadsto m$, we first deduce from the induction assumption \eqref{KritErrRekursion} the inequality
		\[ \|U(t_{N_m})-U_{N_m}\|_{\dot H^1 \times L^2} \le \mathrm{Err}_{m-1} \le 9\eps\sum_{k=0}^{m-1}2^k = 9\eps(2^{m}-1) \le 9\eps(2^L-1). \]
		As in \eqref{t-T_m}, we obtain
		\begin{equation}\label{kleinerrho}
		|T_m - t_{N_m}| = \sum_{j=0}^{m-1} \Big(b_j-\Big\lfloor \frac{b_j}{\tau}\Big\rfloor\tau\Big) \le m\tau \le L\tau \le \rho, 
		\end{equation}
		using also \eqref{KritTau0}.
		Hence, \eqref{KritGlmStet} and \eqref{Eps0} imply
		\begin{align}\label{Ubleibtnah}
		\|&U(T_m)-U_{N_m}\|_{\dot H^1 \times L^2}\nonumber\\
		&\le \|U(T_m)-U(t_{N_m})\|_{\dot H^1 \times L^2} + \|U(t_{N_m})-U_{N_m}\|_{\dot H^1 \times L^2}\nonumber \\
		&\le \eps + 9\eps(2^L-1) \le 9\eps 2^L \le \delta.
		\end{align}
		So we can apply Proposition \ref{PropLokStab} (with $W^0 = U(T_m)$). Furthermore, we write
		\begin{align*}
		\mathrm{Err_m} &= \|U(t_n)-U_n\|_{\ell^\infty_\tau(J_m,\dot H^1 \times L^2)}\\
		&= \|U(t_n,U(t_{N_m}))-\Phi_\tau^n(U_{N_m})\|_{\ell^\infty_\tau([0,b_m],\dot H^1 \times L^2)}. 
		\end{align*}
		In the case $m=L$, we would have to replace the interval $[0,b_m]$ with the interval $[0,T-t_{N_L}]$ (which is smaller by \eqref{t-T_m}), but for simplicity we keep this abuse of notation.
		As noted after \eqref{LieFilt}, we have $U_n = \Pi_{\tau^{-1}} U_n$. Using also \eqref{KritGlmStet} and \eqref{kleinerrho}, we can now estimate similar as for $m=0$ and conclude
		\begin{align*}
		\mathrm{Err_m} &= \|U(t_n,U(t_{N_m}))-\Phi_\tau^n(U_{N_m})\|_{\ell^\infty_\tau([0,b_m],\dot H^1 \times L^2)}\\ 
		&\le \|U(t_n,U(t_{N_m}))-U(t_n,Y^m)\|_{\ell^\infty_\tau([0,b_m],\dot H^1 \times L^2)}\\
		&\quad + \|U(t_n,Y^m)-\Phi_\tau^n(\Pi_{\tau^{-1}} Y^m)\|_{\ell^\infty_\tau([0,b_m],\dot H^1 \times L^2)}\\
		&\quad + \|\Phi_\tau^n(\Pi_{\tau^{-1}} Y^m)-\Phi_\tau^n(\Pi_{\tau^{-1}} U_{N_m})\|_{\ell^\infty_\tau([0,b_m],\dot H^1 \times L^2)}\\
		&\le 2\|U(t_{N_m})-Y^m\|_{\dot H^1 \times L^2} + C(Y^m)\tau + 2\|Y^m-U_{N_m}\|_{\dot H^1 \times L^2}\\
		&\le 2\|U(t_{N_m})-Y^m\|_{\dot H^1 \times L^2} + \eps + 2\|Y^m-U(t_{N_m})\|_{\dot H^1 \times L^2}\\
		&\quad + 2\|U(t_{N_m})-U_{N_m}\|_{\dot H^1 \times L^2}\\
		&\le 4\|U(t_{N_m})-U(T_m)\|_{\dot H^1 \times L^2} + 4\|U(T_m)-Y^m\|_{\dot H^1 \times L^2} + \eps +  2\mathrm{Err}_{m-1}\\
		&\le  9\eps + 2\mathrm{Err}_{m-1}.
		\end{align*}
		Therefore, \eqref{KritErrRekursion} is true. It follows that
		\[\mathrm{Err}_m \le 9\eps \sum_{k=0}^m 2^k = 9\eps (2^{m+1}-1) \le 9\eps(2^{L+1}-1) \]
		for all $m \in \{0,\dots,L\}$. This also shows the convergence in the $\dot H^1 \times L^2$-norm as stated in \eqref{KonvH1}, since $L$ is independent of $\eps$, which we could have replaced by any $\tilde \eps \in (0,\eps]$. To complete the proof of Theorem \ref{ThmKrit}, we will actually not need the full statement of \eqref{KonvH1}. It is enough to know that $\|U(T_m)-U_{N_m}\|_{\dot H^1 \times L^2} \le \delta$ for all $m \in \{0,\dots,L\}$, as noted in \eqref{Ubleibtnah}. \smallskip
		
		\emph{Step 3.} We show the convergence of the scheme in the $L^2 \times\dot H^{-1} $-norm. Recall the notation $E_n = \Pi_{\tau^{-1}} U(t_n)-U_n$. Let $m \in \{0,\dots,L\}$. We use the recursion formula \eqref{FehlerRek} and estimate similar as in the proof of Proposition \ref{PropKonvSubkrit} to obtain
		\begin{align*}
		\|&E_{n}\|_{\ell^\infty_\tau(J_m,L^2 \times \dot H^{-1})} \\
		&= 
		\|E_{N_m+n}\|_{\ell^\infty_\tau([0,b_m],L^2 \times \dot H^{-1})}\\
		&\le \|e^{n\tau A}E_{N_m}\|_{\ell^\infty_\tau([0,b_m],L^2 \times \dot H^{-1})}\\
		&\quad + \Big\|\int_0^{t_n}\Pi_{\tau^{-1}} e^{(n\tau-s)A}B_{N_m}(s) \dd s\Big\|_{\ell^\infty_\tau([0,b_m],L^2 \times \dot H^{-1})}\\
		&\quad +  \Big\|\sum_{k=0}^{n-1}\Pi_{\tau^{-1}} e^{(n-k)\tau A}\Big(\Delta_{{N_m}+k} + \tau Q_{{N_m}+k}\Big)\Big\|_{\ell^\infty_\tau([0,b_m],L^2 \times \dot H^{-1})}\\
		&\le \|E_{N_m}\|_{L^2\times\dot H^{-1}} + \|B_{N_m}\|_{L^1_{b_m}(L^2 \times \dot H^{-1})}+ \tau^{-1}\|\Delta_{N_m+n}\|_{\ell^1_\tau([0,b_m],L^2 \times \dot H^{-1})} \\
		&\quad + \|Q_{N_m+n}\|_{\ell^1_\tau([0,b_m],L^2 \times \dot H^{-1})} \\
		&\le  \|E_{N_m}\|_{L^2\times\dot H^{-1}} + \|g(u)-g(\pi_{\tau^{-1}} u)\|_{L^1_T\dot H^{-1}}+ \tau^{-1} \|\Delta_n\|_{\ell^1_\tau([0,T],L^2 \times \dot H^{-1})} \\
		&\quad + C_\mathrm{So}\|g(\pi_{\tau^{-1}} u(t_{N_m+n}))-g(u_{N_m+n})\|_{\ell^1_\tau([0,b_m], L^{6/5})} \\
		&\le \|E_{N_m}\|_{L^2\times\dot H^{-1}} + 2C_F\tau+ \frac52 C_\mathrm{So}\Big(\|\pi_{\tau^{-1}} u(t_{N_m+n})\|^4_{\ell^4_\tau([0,b_m],L^{12})} \\
		&\quad +  \|u_{N_m+n}\|^4_{\ell^4_\tau([0,b_m],L^{12})}\Big)\|\pi_{\tau^{-1}} u(t_{N_m+n})-u_{N_m+n}\|_{\ell^\infty_\tau([0,b_m],L^{2})} \\
		&\le \|E_{N_m}\|_{L^2\times\dot H^{-1}} + 2C_F\tau+ 5 C_\mathrm{So}\kappa^4\delta^4\|E_n\|_{\ell^\infty_\tau(J_m,L^{2}\times \dot H^{-1})} \\
		&\le \|E_{N_m}\|_{L^2\times\dot H^{-1}} + 2C_F\tau + \frac{1}{2}\|E_n\|_{\ell^\infty_\tau(J_m,L^{2}\times \dot H^{-1})}.
		\end{align*}
		Here we used Lemma \ref{LemB} and \ref{LemLokF} (with constant $C_F$), the bounds from Proposition \ref{PropLokStab} and the definition of $\delta$ in \eqref{Defdelta}. We could apply Proposition \ref{PropLokStab} thanks to the estimates on $U(T_m)$ in \eqref{KritUKlein} and \eqref{SUklein}, on $U_{N_m}-U(T_m)$ in \eqref{Ubleibtnah}, and on $\tau$ in \eqref{KritTau0}.
		The above inequality in display leads to
		\begin{align*}
		\|E_{n}\|_{\ell^\infty_\tau(J_m,L^2 \times \dot H^{-1})} &\le 2\|E_{N_m}\|_{L^2\times\dot H^{-1}} + 4C_F\tau \\
		&\le 2\|E_{n}\|_{\ell^\infty_\tau(J_{m-1},L^2 \times \dot H^{-1})} +  4C_F\tau  
		\end{align*}
		if $m>0$. Since $E_0=0$,
		this recursion formula yields
		the global bound
		\[\|E_n\|_{\ell^\infty_\tau([0,T],L^2 \times \dot H^{-1})} \le 4C_F\tau \sum_{k=0}^{L}2^k = 4C_F(2^{L+1}-1)\tau. \]
		This shows the assertion since we can again use Lemma \ref{LemPi} as in the proof of Theorem \ref{Thm1}.
	\end{proof}
	
	\section{Corrected Lie Splitting}\label{SecKorr}
	In this section we only consider the cubic wave equation, i.e., the case $\alpha = 3$. For simplicity, we set $g(u)\coloneqq -\mu u^3$, one could similarly deal with $g(u)= -\mu |u|^2u$ in case of complex-valued functions. We prove an error estimate with order $3/2$ for a frequency-filtered variant of the corrected Lie splitting recently proposed in \cite{KleinGordon}. The original form of the method reads
	\[U_{n+1} = e^{\tau A} [U_n+\tau G(U_n) + \tau^2 \varphi_2(-2\tau A)H(U_n)], \] 
	where the operator $A$ and the nonlinearity $G$ are defined in \eqref{Defs}. We thus have added a correction term to the Lie splitting. Here we set
	\begin{equation}\label{Defphi}
	H(u,v) \coloneqq \begin{pmatrix}-g(u) \\ g'(u)v\end{pmatrix} \quad \text{and} \quad \varphi_2(t A)w \coloneqq \int_0^1 (1-\sigma)e^{\sigma t A} w \dd \sigma  
	\end{equation}
	for $t \in \R$.
	Since $v$ corresponds to $\partial_tu$, the nonlinear term $H$ now contains a derivative.
	The operator $\varphi_2(t A)$ can also be understood by the functional calculus for $A$ and the function $\varphi_2(z) = (e^z-z-1)/z^2$, which is bounded on $\iu \R$. The expression \eqref{WaveGroup} for the wave group $e^{tA}$ leads to the formula
	\[\varphi_2(tA) = \begin{pmatrix}
	\dfrac{1-\cos(t|\nabla|)}{t^2|\nabla|^2} & \dfrac{t|\nabla|-\sin(t|\nabla|)}{t^2|\nabla|^3} \\[2ex]
	\dfrac{\sin(t|\nabla|)-t|\nabla|}{t^2|\nabla|} & \dfrac{1-\cos(t|\nabla|)}{t^2|\nabla|^2}
	\end{pmatrix}.\]
	This implies a smoothing property for the operator $\varphi_2(tA)$. Using the functional calculus for $t|\nabla|$, we deduce the smoothing property 
	\begin{equation}\label{phi2Absch}
	\Big\|\varphi_2(tA)\begin{pmatrix}z \\ 0\end{pmatrix}\Big\|_{L^2 \times \dot H^{-1}} \lesssim |t|^{-1} \|z\|_{\dot H^{-1}}
	\end{equation}
	of $\varphi_2(tA)$ that is valid for all $z \in \dot H^{-1}$ and $t \neq 0$.
	
	As explained in \cite{KleinGordon}, the corrected Lie splitting is formally of second order, but, in contrast to classical second-order integrators for the wave equation, only first-order spatial derivatives appear in the local error. In \cite{KleinGordon}, a second-order error estimate in $H^1 \times L^2$ was shown for solutions $u$ with $H^{1+d/4}$ regularity, where $d$ denotes the spatial dimension. Since here, we only assume $\dot H^1$ regularity of the solution, we have to use the discrete Strichartz estimates from Section \ref{SecStrich} to deal with the error terms. Therefore, as in Section \ref{SecLie}, we include the frequency filter $\pi_K$ from \eqref{DefPi} in the scheme, which gives
	\begin{equation}\label{KorrLie}\begin{aligned}
	U_{n+1} &= e^{\tau A}\Pi_K[U_n + \tau G(U_n) + \tau^2 \varphi_2(-2\tau A)H(U_n) ],\\
	U_0 &= \Pi_K U^0. 
	\end{aligned}
	\end{equation}
	Here, we again set
	\[ \Pi_K \coloneqq \begin{pmatrix}\pi_K & 0 \\ 0 & \pi_K \end{pmatrix}. \]
	We write $u_n$ (resp.\ $v_n$) for the first (resp.\ second) component of $U_n$. In contrast to the previous sections, we do not set $K=\tau^{-1}$, since such a choice could only lead to an estimate of order one in $L^2 \times H^{-1}$, because the error coming from frequency truncation would dominate (of order $1/K$, see Lemma \ref{LemPi}). To get an improved convergence, we therefore have to choose some $K>\tau^{-1}$. As seen in Section \ref{SecStrich}, this leads to a multiplicative loss of the form $(K\tau)^{1/p}$ in the discrete Strichartz estimates. Therefore, we cannot reach the optimal order two for the scheme \eqref{KorrLie}. It turns out that with the choice
	\[ K \sim \tau^{-\frac32}\]
	we can optimize the global error and get an error bound of order $3/2$. See \cite{ORS} for a similar discussion in the context of Schrödinger equations.
	
	For the rest of this section, we fix  $K = \tau^{-3/2}$. As in case of the Lie splitting, we define the error terms
	\begin{equation}\label{DefErrorKorr}
	E_n \coloneqq \Pi_K U(t_n) - U_n 
	\end{equation}	
	for all $n \in \N_0$ with $t_n \le T$. For the first component of $E_n$ we write $e_n$. We first show an error recursion for \eqref{DefErrorKorr} which is similar to \eqref{FehlerRek}.
	\begin{Proposition}\label{PropFehlerRekKorr}
		Let $U$ and $T$ be given by Assumption \ref{Ass} with $\alpha=3$ and let $K=\tau^{-3/2}$. The error defined in \eqref{KorrLie} and \eqref{DefErrorKorr} then satisfies
		\begin{align}\label{FehlerRekKorr}
		E_{m+n} &= e^{n\tau A}E_m + \int_0^{t_n}\Pi_K e^{(n\tau-s)A}B_m(s) \dd s \nonumber\\
		&\quad+  \sum_{k=0}^{n-1}\Pi_K e^{(n-k)\tau A}\Big(\widetilde \Delta_{m+k} + \tau \widetilde Q_{m+k}\Big) 
		\end{align}
		for all $\tau \in (0,1]$, and $n, m \in \N_0$ with $t_{m+n} \le T$. The appearing terms are given by 
		\begin{equation}
		\begin{aligned}\label{DefTermeKorr}
		\widetilde \Delta_n &\coloneqq \Delta_n - \tau^2\varphi_2(-2\tau A) H(\Pi_KU(t_n)),\\
		\widetilde Q_n &\coloneqq Q_n + \tau\varphi_2(-2\tau A)\big[H(\Pi_K U(t_n))-H(U_n)\big], 
		\end{aligned}
		\end{equation}
		where $B_m(s)$, $\Delta_n$ and $Q_n$ are defined as in \eqref{DefTerme} with $\Pi_{\tau^{-1}}$ replaced by $\Pi_K$, and $H$ and $\varphi_2(tA)$ are introduced in \eqref{Defphi}.
	\end{Proposition}
	\begin{proof}
		For the discrete approximation $U_n$ defined by \eqref{KorrLie}, we have the discrete Duhamel formula
		\begin{equation*}\label{DiskrDuhamelKorr}
		U_{m+n} = e^{n\tau A}U_m + \tau \sum_{k=0}^{n-1}\Pi_K e^{(n-k)\tau A} \Big(G(U_{m+k})+\tau\varphi_2(-2\tau A)H(U_{m+k})\Big).
		\end{equation*}
		Proceeding as in Proposition \ref{PropFehlerRek}, we derive
		\begin{align*}
		E_{m+n} &=\Pi_K U(t_{m+n}) - U_{m+n}\\
		&= e^{n\tau A} (\Pi_K U(t_m)-U_m) + \int_0^{t_n} \Pi_K e^{(n\tau-s)A}G(U(t_m+s)) \dd s\\
		&\quad - \tau \sum_{k=0}^{n-1}\Pi_K e^{(n-k)\tau A} \Big(G(U_{m+k})+\tau\varphi_2(-2\tau A)H(U_{m+k})\Big)\\
		&= e^{n\tau A} E_m + \int_0^{t_n}\Pi_K e^{(n\tau-s)A}B_m(s) \dd s\\
		&\quad + \sum_{k=0}^{n-1}\Pi_K e^{(n-k)\tau A}\Big(\Delta_{m+k} + \tau Q_{m+k}\Big)\\
		&\quad - \tau^2 \sum_{k=0}^{n-1}\Pi_K e^{(n-k)\tau A} \varphi_2(-2\tau A)H(U_{m+k})\\
		&= e^{n\tau A} E_m + \int_0^{t_n}\Pi_K e^{(n\tau-s)A}B_m(s) \dd s\\
		&\quad +  \sum_{k=0}^{n-1}\Pi_K e^{(n-k)\tau A}\Big(\widetilde \Delta_{m+k} + \tau \widetilde Q_{m+k}\Big). \qedhere
		\end{align*}
	\end{proof}
	
	As we see in the next lemma, in the error formula \eqref{FehlerRekKorr} all second derivatives of $u$ cancel, only first-order derivatives remain. This is the advantage of the corrected Lie splitting in comparison with classical second-order integrators, such as the Strang splitting. We first deal with the local error $\widetilde \Delta_n$.
	
	\begin{Lemma}\label{LemLokFKorr}
		Let $U=(u,\partial_tu)$, $T$, and $M$ be given by Assumption \ref{Ass} with $\alpha=3$ and let $K=\tau^{-3/2}$. We then have the representation
		\begin{equation}\label{ErrReprKorr}
		\widetilde \Delta_n = \int_0^\tau(\tau-\sigma)e^{-2\sigma A} \int_0^\sigma e^{s A}\begin{pmatrix}0 \\ d_1(\tau,s,n)+d_2(\tau,s,n)\end{pmatrix} \dd s\dd \sigma, 
		\end{equation}
		where we abbreviate
		\begin{align*}
		d_1(\tau,s,n) &\coloneqq g''(\pi_K u(t_n+s))\big[(\pi_K \partial_t u(t_n+s))^2-(\nabla \pi_Ku(t_n+s))^2\big], \\
		d_2(\tau,s,n) &\coloneqq g'(\pi_K u(t_n+s))\pi_K g(u(t_n+s)).
		\end{align*}
		Here we use the notation $(\nabla f)^2 \coloneqq \sum_{j=1}^3 (\partial_j f)^2$. Moreover, the inequalities
		\begin{align*}
		\Big\|\sum_{k=0}^{N-1}\Pi_K e^{(N-k)\tau A}\widetilde \Delta_{m+k}\Big\|_{L^2\times \dot H^{-1}} &\lesssim_{M,T} \tau^{\frac32}, \\
		\Big\|\sum_{k=0}^{n-1} \pi_K S((n-k)\tau)\widetilde \Delta_{m+k}\Big\|_{\ell^p_{\tau,N}L^q} &\lesssim_{M,T,p,q} \tau^{\frac32(1-\gamma)-\frac{1}{2p}}
		\end{align*}
		hold for all $\tau \in (0,1]$, $N \in \N$, $m \in \N_0$ with $t_{m+N} \le T$, and wave admissible parameters $(p,q,\gamma)$. 
	\end{Lemma}
	\begin{proof}
		The error representation of the Lie splitting \eqref{ErrRepres} and Fubini's theorem yield
		\begin{align*}
		\Delta_n &= \int_0^\tau \int_0^s e^{-\sigma A}H(\Pi_KU(t_n+\sigma)) \dd \sigma\dd s \\
		&= \int_0^\tau (\tau-\sigma) e^{-\sigma A}H(\Pi_KU(t_n+\sigma)) \dd \sigma.  
		\end{align*}
		Moreover, the definition of $\varphi_2$ in \eqref{Defphi} and a substitution lead to
		\begin{align*}
		\tau^2 \varphi_2(-2\tau A)H(\Pi_K U(t_n)) &= \tau^2 \int_0^1 (1-\sigma)e^{-2\tau\sigma A} H(\Pi_K U(t_n)) \dd \sigma \\
		&= \int_0^\tau (\tau-\sigma)e^{-2\sigma A} H(\Pi_K U(t_n)) \dd \sigma.
		\end{align*}
		From \eqref{DefTermeKorr} it thus follows
		\begin{align*}
		\widetilde \Delta_n &= \Delta_n - \tau^2\varphi_2(-2\tau A) H(\Pi_KU(t_n))\\
		&= \int_0^\tau (\tau-\sigma)e^{-2\sigma A}\Big[e^{\sigma A}H(\Pi_KU(t_n+\sigma))-H(\Pi_K U(t_n))\Big] \dd \sigma \\
		&= \int_0^\tau (\tau-\sigma)e^{-2\sigma A} \int_0^\sigma \frac{\mathrm{d}}{\mathrm{d}s}\Big[e^{s A}H(\Pi_KU(t_n+s))\Big] \dd s \dd \sigma.
		\end{align*}	
		We compute the derivative by
		\begin{align*}
		&\frac{\mathrm{d}}{\mathrm{d}s}\Big[e^{s A}H(\Pi_KU(t_n+s))\Big]\\
		&= \frac{\mathrm{d}}{\mathrm{d}s}\Bigg[e^{s A}\begin{pmatrix}-g(\pi_K u(t_n+s)) \\ g'(\pi_K u(t_n+s))\pi_K \partial_t u(t_n+s)\end{pmatrix}\Bigg] \\
		&= e^{sA}\Bigg[\begin{pmatrix}0 & I \\ \Delta & 0 \end{pmatrix}\begin{pmatrix}-g(\pi_K u(t_n+s)) \\ g'(\pi_K u(t_n+s))\pi_K \partial_t u(t_n+s)\end{pmatrix} \\ 
		&\quad +\begin{pmatrix}-g'(\pi_K u(t_n+s))\pi_K\partial_tu(t_n+s) \\ g''(\pi_K u(t_n+s))(\pi_K \partial_t u(t_n+s))^2+g'(\pi_K u(t_n+s))\pi_K \partial_{tt} u(t_n+s) \end{pmatrix} \Bigg] \\
		&= e^{sA}\begin{pmatrix}0 \\ d_1(\tau,s,n)+g'(\pi_K u(t_n+s))\pi_K [\partial_{tt} u(t_n+s)-\Delta u(t_n+s)] \end{pmatrix} \\
		&= e^{sA}\begin{pmatrix}0 \\ d_1(\tau,s,n)+d_2(\tau,s,n) \end{pmatrix}, 
		\end{align*}
		where we used the formula $\Delta[g(w)] = g''(w)(\nabla w)^2 + g'(w) \Delta w$ and the fact that the differential equation \eqref{NLW2} holds in $C([0,T],H^{-1})$. Because of $\pi_K$ there is no problem in justifying the above differentiations. Thus, the error representation formula \eqref{FehlerRekKorr} is true.
		
		For the error estimates, we first prove the inequalities
		\begin{align}\label{d1Absch}
		\sup_{s \in [0,\tau]}\|d_1(\tau,s,n)\|_{\ell^2_{\tau,N-1}L^1} &\lesssim_{M,T} \tau^{-\frac14},\\
		\label{d2Absch}\sup_{s \in [0,\tau]}\|d_2(\tau,s,n)\|_{\ell^1_{\tau,N-1}\dot H^{-1}} &\lesssim_{M,T} 1.
		\end{align}
		Note that we have set $g(u) = -\mu u^3$, hence $g'(u) = -3\mu u^2$ and $g''(u) = -6\mu u$. We derive 
		\begin{align*}
		&\|d_1(\tau,s,n)\|_{\ell^2_{\tau,N-1}L^1}\\ 
		&\le \|g''(\pi_K u(t_n\!+s))\|_{\ell^2_{\tau,N-1}L^\infty}\|(\pi_K \partial_t u(t_n+s))^2-(\nabla \pi_Ku(t_n+s))^2\|_{\ell^\infty_{\tau,N-1}L^1} \\
		&\lesssim \|\pi_K u(t_n+s)\|_{\ell^2_{\tau,N-1}L^\infty}\Big(\|\pi_K \partial_t u\|^2_{L^\infty_TL^2}+\||\nabla \pi_Ku|\|^2_{L^\infty_TL^2}\Big) \\
		&\lesssim_{M,T} (K\tau+\log K)^{\frac12} \lesssim \tau^{-\frac14},
		\end{align*}
		uniformly in $s \in [0,\tau]$, where we used Proposition \ref{PropDiskrStrichU} and the relation $K = \tau^{-3/2}$. The second term is easier to estimate, we can directly infer
		\begin{align*}
		\|d_2(\tau,s,n)\|_{\ell^1_{\tau,N-1}\dot H^{-1}} &\lesssim \|g'(\pi_K u(t_n+s))\pi_K g(u(t_n+s))\|_{\ell^1_{\tau,N-1}L^{6/5}} \\
		&\lesssim_T \|g'(\pi_K u)\|_{L^\infty_TL^3}\|\pi_K g(u)\|_{L^\infty_TL^2}\\
		&\lesssim \|(\pi_K u)^2\|_{L^\infty_TL^3}\|u^3\|_{L^\infty_TL^2}\\
		&=\|\pi_Ku\|^2_{L^\infty_TL^6}\|u\|^3_{L^\infty_TL^6} \lesssim \|u\|^5_{L^\infty_T\dot H^1} \lesssim_M 1.
		\end{align*}
		This shows \eqref{d1Absch} and \eqref{d2Absch}. 
		
		Now we turn to the error estimates, starting with the $L^2 \times \dot H^{-1}$-norm. Formulas \eqref{ErrReprKorr} and \eqref{WaveGroup}, the endpoint estimate \eqref{EndpDual} and the inequalities \eqref{d1Absch} and \eqref{d2Absch} imply
		\begin{align*}
		\Big\|&\sum_{k=0}^{N-1}\Pi_K e^{(N-k)\tau A}\widetilde \Delta_{m+k}\Big\|_{L^2\times \dot H^{-1}} \\
		&\le \tau^3 \sup_{s \in [0,\tau]} \Big\|\sum_{k=0}^{N-1}\Pi_Ke^{-k\tau A}\begin{pmatrix}0 \\ d_1(\tau,s,m+k)+d_2(\tau,s,m+k)\end{pmatrix} \Big\|_{L^2\times \dot H^{-1}} \\
		&\lesssim \tau^2 \sup_{s \in [0,\tau]}\Big(\Big\|\tau\sum_{k=0}^{N-1} \pi_Ke^{\pm \iu k\tau |\nabla|} d_1(\tau,s,m+k) \Big\|_{\dot H^{-1}} \\
		&\quad + \|d_2(\tau,s,m+n)\|_{\ell^1_{\tau,N-1}\dot H^{-1}}\Big) \\
		&\lesssim_{M,T} \tau^2 \sup_{s \in [0,\tau]}\Big((K\tau+\log K)^{\frac12}\|d_1(\tau,s,m+n) \|_{\ell^2_{\tau,N-1}L^1} + 1\Big) \\
		&\lesssim_{M,T}\tau^{2-\frac14-\frac14} = \tau^{\frac32}.
		\end{align*}
		For the estimate in the $\ell^p_\tau L^q$-norm we can proceed similarly, where we now employ Theorem \ref{ThmDisStrich} for the $d_2$-term. Moreover, for the inequality involving $d_1$ we have to use Corollary \ref{KorInhomCK}, since we have non-$L^2$-based spaces in the spatial variable on both sides of the inhomogeneous Strichartz estimate. This gives
		\begin{align*}
		\Big\|&\sum_{k=0}^{n-1} \pi_K S((n-k)\tau)\widetilde \Delta_{m+k}\Big\|_{\ell^p_{\tau,N}L^q} \\
		&\le \tau^2 \sup_{s,\sigma \in [0,\tau]} \Big\|\tau\sum_{k=0}^{n-1} \pi_KS((n-k)\tau+s-2\sigma)\\
		&\qquad \qquad \qquad \cdot \begin{pmatrix}0 \\ d_1(\tau,s,m+k)+d_2(\tau,s,m+k)\end{pmatrix} \Big\|_{\ell^p_{\tau,N}L^q} \\
		&\lesssim_{p,q,T} \tau^2 (K\tau)^{\frac1p}K^\gamma\sup_{s \in [0,\tau]}\Big((K\tau+\log K)^{\frac12}\|d_1(\tau,s,m+n) \|_{\ell^2_{\tau,N-1}L^1} \\
		&\qquad \qquad \qquad \qquad \qquad \quad+ \|d_2(\tau,s,m+n)\|_{\ell^1_{\tau,N-1}\dot H^{-1}}\Big) \\
		&\lesssim_{M,T} \tau^{2-\frac{1}{2p}-\frac32 \gamma}\Big(\tau^{-\frac12}+1\Big) \lesssim \tau^{\frac32(1-\gamma)-\frac{1}{2p}},
		\end{align*}
		as desired. 
	\end{proof}
	
	We further need a product estimate in $\dot H^{-1}$.
	
	\begin{Lemma}\label{LemProd}
		The inequality
		\[ \|vw\|_{\dot H^{-1}} \lesssim \|v\|_{\dot H^1}\|w\|_{\dot H^{-\frac12}}  \]
		holds for all  $v \in \dot H^1$ and $w \in \dot H^{-\frac12} \cap L^1_{\mathrm{loc}}$.
	\end{Lemma}
	\begin{proof}
		We use the duality between $\dot H^s$ and $\dot H^{-s}$ for $|s| < 3/2$. We estimate
		\begin{align*}
		\|vw\|_{\dot H^{-1}} = \sup_{\substack{z \in \dot H^1 \\ \|z\|_{\dot H^1 \le 1}}} \Big| \int_{\R^3} vwz \dd x \Big| \le \sup_{\substack{z \in \dot H^1 \\ \|z\|_{\dot H^1 \le 1}}} \|vz\|_{\dot H^{\frac12}}\|w\|_{\dot H^{-\frac12}}. 
		\end{align*}
		Next, we use the fractional product rule
		\[\||\nabla|^a (vz)\|_{L^r} \lesssim \||\nabla|^a v\|_{L^{p_1}} \|z\|_{L^{q_2}} + \|v\|_{L^{p_1}} \||\nabla|^a z\|_{L^{q_2}} \]
		for $a \in (0,1)$ and $p_j$, $q_j$, $r \in (1,\infty)$ with $\frac{1}{p_1}+\frac{1}{q_1}=\frac{1}{p_2}+\frac{1}{q_2} = \frac{1}{r}$, see Proposition 3.3 in \cite{ChristWeinstein}. It follows
		\begin{align*}
		\|vz\|_{\dot H^{\frac12}} &\lesssim \||\nabla|^{\frac12}v\|_{L^3}\|z\|_{L^6} + \|v\|_{L^6}\||\nabla|^{\frac12}z\|_{L^3}\\
		&\lesssim \||\nabla|^{\frac12}v\|_{\dot H^{\frac12}}\|z\|_{\dot H^1} + \|v\|_{\dot H^1}\||\nabla|^{\frac12}z\|_{\dot H^{\frac12}} \\
		&\lesssim \|v\|_{\dot H^1} \|z\|_{\dot H^{1}},
		\end{align*}
		using also the Sobolev embeddings $\dot H^1 \hookrightarrow L^{6}$ and $\dot H^{\frac12} \hookrightarrow L^3$.
	\end{proof}
	
	To treat the error term $\widetilde Q_n$, we use \eqref{DefTermeKorr} to decompose
	\begin{equation}\label{QDec}
	\widetilde Q_n = \begin{pmatrix}0 \\ \mathbf q_1(\tau,n) \end{pmatrix} + \tau \varphi_2(-2\tau A) \begin{pmatrix} -\mathbf q_1(\tau,n) \\ \mathbf q_2(\tau,n) + \mathbf q_3(\tau,n) \end{pmatrix}
	\end{equation}
	with the definitions
	\begin{align*}
	\mathbf q_1(\tau,n) &\coloneqq g(\pi_K u (t_n))-g(u_n), \\
	\mathbf q_2(\tau,n) &\coloneqq g'(\pi_K u (t_n))[\pi_K \partial_tu (t_n)-v_n], \\
	\mathbf q_3(\tau,n) &\coloneqq v_n[g'(\pi_K u (t_n))-g'(u_n)],
	\end{align*}
	where $U_n=(u_n,v_n)$. This time we need to measure the error simultaneously in three different norms, namely
	\begin{equation}\label{DefNormKorr}
	\opnorm{E_n}_{\tau,j} \coloneqq \max\Big\{\tau^{-\frac32}\|E_n\|_{\ell^\infty_{\tau,j}(L^2 \times  \dot H^{-1})}, \tau^{-\frac14}\|e_n\|_{\ell^4_{\tau,j}L^{6}}, \tau^{\frac18}\|e_n\|_{\ell^4_{\tau,j}L^{12}}\Big\},
	\end{equation}
	where $j \in \N_0$ is a number with $j\tau \le T$. The rates are consistent with those obtained in Lemma \ref{LemLokFKorr}, since the parameters $(4,6,\frac34)$ and $(4,12,1)$ are wave admissible. 
	
	\begin{Lemma}\label{LemQ}
		Let $u$, $T$, and $M$ be given by Assumption \ref{Ass} with $\alpha=3$, and let $K=\tau^{-3/2}$. Define the error $E_n$ by \eqref{KorrLie} with $K=\tau^{-3/2}$ and \eqref{DefErrorKorr}.
		Then the estimates
		\begin{align*}
		\|\mathbf q_1(\tau,m+n)\|_{\ell^1_{\tau,j} \dot H^{-1}} &\lesssim_{M,T} t_{j+1}^{\frac12}\tau^{\frac32} \Big(1+\tau^{\frac12}\opnorm{E_{m+n}}^2_{\tau,j}\Big)\opnorm{E_{m+n}}_{\tau,j} \\
		\|\mathbf q_2(\tau,m+n)\|_{\ell^1_{\tau,j} \dot H^{-1}} &\lesssim_{M,T} t_{j+1}^{\frac12} \tau^{\frac12} \opnorm{E_{m+n}}_{\tau,j} \\	
		\|\mathbf q_3(\tau,m+n)\|_{\ell^2_{\tau,j} L^1} &\lesssim_{M,T} t_{j+1}^{\frac14} \tau\Big(\opnorm{E_{m+n}}_{\tau,j}+\opnorm{E_{m+n}}_{\tau,j}^3\Big)
		\end{align*}
		hold for all $\tau \in (0,1]$ and $m, j \in \N_0$ with $(m+j)\tau \le T$.
	\end{Lemma}
	\begin{proof}
		For simplicity, we set $m=0$ in the proof of these three estimates, since the shift by $m>0$ does not affect the argument. The first estimate is the same as in Lemma \ref{LemStab}, now with $\alpha=3$, implying
		\begin{align*}
		\|\mathbf q_1(\tau,n)\|_{\ell^1_{\tau,j} \dot H^{-1}} &= \|g(\pi_K u(t_{n}))-g(u_{n})\|_{\ell^1_{\tau,j}\dot H^{-1}}\\
		&\lesssim_{M,T} t_{j+1}^{\frac12}\Big(1+\|e_{n}\|^{2}_{\ell^4_{\tau,j}L^{6}}\Big)\|e_{n}\|_{\ell^\infty_{\tau,j}L^2}\\
		&\le t_{j+1}^{\frac12}\tau^{\frac32} \Big(1+\tau^{\frac12}\opnorm{E_{n}}^2_{\tau,j}\Big)\opnorm{E_{n}}_{\tau,j}.
		\end{align*}
		For the second inequality, we use the product estimate from Lemma \ref{LemProd} and note that $g'(u) = -3\mu u^2$. We compute
		\begin{align*}
		\|&\mathbf q_2(\tau,n)\|_{\ell^1_{\tau,j} \dot H^{-1}} \\
		&\lesssim \|g'(\pi_K u (t_n))\|_{\ell^1_{\tau,j}\dot H^1}\|\pi_K \partial_tu (t_n)-v_n\|_{\ell^\infty_{\tau,j}\dot H^{-\frac12}}\\
		&\lesssim \|\pi_K u (t_n)\|_{\ell^2_{\tau,j}L^\infty} \|\pi_K u (t_n)\|_{\ell^2_{\tau,j}\dot H^1}\tau^{-\frac32\cdot\frac12}\|\pi_K \partial_tu (t_n)-v_n\|_{\ell^\infty_{\tau,j}\dot H^{-1}} \\
		&\lesssim_{M,T} t_{j+1}^{\frac12}\tau^{-\frac14+\frac34}\|\pi_K u (t_n)\|_{\ell^\infty_{\tau,j}\dot H^1}\opnorm{E_n}_{\tau,j} \\
		&\lesssim_M t_{j+1}^{\frac12}\tau^{\frac12}\opnorm{E_n}_{\tau,j},
		\end{align*}
		employing also Proposition \ref{PropDiskrStrichU} for the estimate in $\ell^2_\tau L^\infty$ and Bernstein's inequality to change from $\dot H^{-1/2}$ to $\dot H^{-1}$. Recall that $E_n = \Pi_K E_n$  and $K = \tau^{-3/2}$. The third inequality is shown similarly by
		\begin{align*}
		\|&\mathbf q_3(\tau,n)\|_{\ell^2_{\tau,j} L^1} \\
		&\lesssim \|v_n\|_{\ell^\infty_{\tau,j}L^2}\|g'(\pi_K u (t_n))-g'(u_n)\|_{\ell^2_{\tau,j} L^2} \\
		&\lesssim_M \Big(1 + \|\pi_K\partial_t u(t_n)-v_n\|_{\ell^\infty_{\tau,j}L^2}\Big)   \Big(\|\pi_Ku(t_n)\|_{\ell^2_{\tau,j}L^{12}}+\|u_n\|_{\ell^2_{\tau,j}L^{12}}\Big)\\
		&\quad \cdot\|e_n\|_{\ell^\infty_{\tau,j}L^\frac{12}{5}} \\
		&\lesssim \Big(1+\tau^{-\frac32}\|\pi_K\partial_t u(t_n)-v_n\|_{\ell^\infty_{\tau,j}\dot H^{-1}}\Big)   t_{j+1}^{\frac14} \Big(\|\pi_Ku(t_n)\|_{\ell^{4}_{\tau,j}L^{12}}+\|u_n\|_{\ell^{4}_{\tau,j}L^{12}}\Big)\\
		&\quad \cdot\|e_n\|_{\ell^\infty_{\tau,j}\dot H^{\frac14}} \\
		&\lesssim_{M,T} t_{j+1}^{\frac14} \Big(1+\opnorm{E_n}_{\tau,j})\tau^{-\frac18}\Big(1+\opnorm{E_n}_{\tau,j}\Big)\tau^{\frac{3}{2}(1-\frac14)}\opnorm{E_n}_{\tau,j} \\
		&\lesssim t_{j+1}^{\frac14} \tau\Big(\opnorm{E_n}_{\tau,j}+\opnorm{E_n}_{\tau,j}^3\Big),
		\end{align*}
		where also the Sobolev embedding $\dot H^{1/4} \hookrightarrow L^{12/5}$ was used.
	\end{proof}

	We can now estimate the remaining part of the error formula \eqref{FehlerRekKorr}.
	
	\begin{Lemma}\label{LemStabKorr}
		Let $U=(u,\partial_tu)$, $T$, and $M$ be given by Assumption \ref{Ass} with $\alpha=3$, and let $K=\tau^{-3/2}$. Define the error by \eqref{KorrLie} with $K=\tau^{-3/2}$ and \eqref{DefErrorKorr}.
		Then the inequality
		\begin{align*}
		\opnorm[\Big]{&\tau\sum_{k=0}^{n-1}\Pi_Ke^{(n-k)\tau A}\widetilde Q_{m+k}}_{\tau,j+1}\\
		&\lesssim_{M,T}  \max\{t_{j+1}^{\frac14},t_{j+1}^{\frac12}\}\Big(\opnorm{E_{m+n}}_{\tau,j}+\tau^{\frac14}  \opnorm{E_{ m+n}}^3_{\tau,j}\Big) 
		\end{align*}
		holds for all $\tau \in (0,1]$ and $m, j \in \N_0$ with $(m+j+1)\tau \le T$. 
	\end{Lemma}
	\begin{proof}
		We first estimate in the $\ell^\infty_\tau(L^2 \times \dot H^{-1})$-norm by means of the decomposition \eqref{QDec}, obtaining
		\begin{align}\label{QtildeinL2H-1}
		\Big\|&\tau\sum_{k=0}^{n-1}\Pi_Ke^{(n-k)\tau A}\widetilde Q_{m+k}\Big\|_{\ell^\infty_{\tau,j+1}(L^2 \times \dot H^{-1})} \nonumber\\
		&\le
		\|\mathbf q_1(\tau,m+n)\|_{\ell^1_{\tau,j}\dot H^{-1}} + \Big\|\tau \varphi_2(-2\tau A) \begin{pmatrix} \mathbf q_1(\tau,m+n) \\ 0 \end{pmatrix}\Big\|_{\ell^1_{\tau,j}(L^2 \times \dot H^{-1})}  \nonumber\\
		&\quad + \tau\| \mathbf q_2(\tau,m+n)\|_{\ell^1_{\tau,j}\dot H^{-1}}\nonumber\\
		&\quad+  \Big\|\tau^2\sum_{k=0}^{n-1}\Pi_Ke^{-k\tau A}\begin{pmatrix} 0 \\ \mathbf q_3(\tau,m+k) \end{pmatrix}\Big\|_{\ell^\infty_{\tau,j+1}(L^2 \times \dot H^{-1})} \nonumber\\
		&\lesssim_T \|\mathbf q_1(\tau,m+n)\|_{\ell^1_{\tau,j}\dot H^{-1}} + \tau\| \mathbf q_2(\tau,m+n)\|_{\ell^1_{\tau,j}\dot H^{-1}}\nonumber\\
		&\quad\quad + \tau^{\frac34} 	\| \mathbf q_3(\tau,m+n)\|_{\ell^2_{\tau,j} L^1} \nonumber\\
		&\lesssim_{M,T} t_{j+1}^{\frac12}\tau^{\frac32} \Big(1+\tau^{\frac12}\opnorm{E_{m+n}}^2_{\tau,j}\Big)\opnorm{E_{m+n}}_{\tau,j} + t_{j+1}^{\frac12} \tau^{\frac32} \opnorm{E_{m+n}}_{\tau,j} \nonumber\\
		&\quad\quad\quad + t_{j+1}^{\frac14} \tau^{\frac74}\Big(\opnorm{E_{m+n}}_{\tau,j}+\opnorm{E_{m+n}}_{\tau,j}^3\Big) \nonumber\\
		&\lesssim \tau^{\frac32}\max\{t_{j+1}^{\frac14},t_{j+1}^{\frac12}\}\Big(\opnorm{E_{m+n}}_{\tau,j}+\tau^{\frac14}  \opnorm{E_{m+n}}^3_{\tau,j}\Big).
		\end{align}
		Here we use inequality \eqref{phi2Absch} for the second term involving $\mathbf q_1$ and the endpoint estimate \eqref{EndpDual} for the term involving $\mathbf q_3$. In the end, the bounds from Lemma \ref{LemQ} were applied.
		
		Let now $(p,q,\gamma)$ be wave admissible. The definition of $\varphi_2$ in \eqref{Defphi} and Corollary \ref{KorInhomCK} lead to
		\begin{align*}
		\Big\|&\tau^2\sum_{k=0}^{n-1} \pi_KS((n-k)\tau) \varphi_2(-2\tau A) \begin{pmatrix} 0 \\ \mathbf q_3(\tau,m+k) \end{pmatrix} \Big\|_{\ell^p_{\tau,j+1}L^q} \\
		&= \Big\|\int_0^1 (1-\sigma)\tau^2\sum_{k=0}^{n-1} \pi_KS((n-k-2\sigma)\tau) \begin{pmatrix} 0 \\ \mathbf q_3(\tau,m+k) \end{pmatrix} \dd \sigma \Big\|_{\ell^p_{\tau,j+1}L^q}  \\ 
		&\le \sup_{\sigma \in [0,1]} \Big\|\tau^2\sum_{k=0}^{n-1} \pi_KS((n-k-2\sigma)\tau) \begin{pmatrix} 0 \\ \mathbf q_3(\tau,m+k) \end{pmatrix} \Big\|_{\ell^p_{\tau,j+1}L^q} \\
		&\lesssim_{p,q,T} \tau(K\tau)^{\frac1p} K^\gamma (K\tau + \log K)^{\frac12} \|\mathbf q_3(\tau,m+n)\|_{\ell^2_{\tau,j}L^1} \\
		&\lesssim \tau^{\frac34-\frac{1}{2p}-\frac32\gamma} \|\mathbf q_3(\tau,m+n)\|_{\ell^2_{\tau,j}L^1}.
		\end{align*} 
		The other terms can be estimated in $\ell^p_\tau L^q$ as in \eqref{QtildeinL2H-1}, now using Theorem \ref{ThmDisStrich}. Altogether, we obtain
		\begin{align*}
		&\tau^{\frac{1}{2p}+\frac32\gamma}\Big\|\tau\sum_{k=0}^{n-1} \pi_KS((n-k)\tau)\widetilde Q_{m+k} \Big\|_{\ell^p_{\tau,j+1}L^q} \\
		&\lesssim_{p,q,T}
		\|\mathbf q_1(\tau,m+n)\|_{\ell^1_{\tau,j}\dot H^{-1}} + \Big\|\tau \varphi_2(-2\tau A) \begin{pmatrix} \mathbf q_1(\tau,m+n) \\ 0 \end{pmatrix}\Big\|_{\ell^1_{\tau,j}(L^2 \times \dot H^{-1})}  \\
		&\quad\ + \tau\| \mathbf q_2(\tau,m+n)\|_{\ell^1_{\tau,j}\dot H^{-1}}+ \tau^{\frac34} \|\mathbf q_3(\tau,m+n)\|_{\ell^2_{\tau,j}L^1} \\
		&\lesssim \|\mathbf q_1(\tau,m+n)\|_{\ell^1_{\tau,j}\dot H^{-1}} + \tau\| \mathbf q_2(\tau,m+n)\|_{\ell^1_{\tau,j}\dot H^{-1}} + \tau^{\frac34} 	\| \mathbf q_3(\tau,m+n)\|_{\ell^2_{\tau,j} L^1} \\
		&\lesssim \tau^{\frac32}\max\{t_{j+1}^{\frac14},t_{j+1}^{\frac12}\}\Big(\opnorm{E_{m+n}}_{\tau,j}+\tau^{\frac14}  \opnorm{E_{m+n}}^3_{\tau,j}\Big).
		\end{align*} 
		The claim now follows from the definition \eqref{DefNormKorr} of the $\opnorm{\cdot}_{\tau,j}$-norm.
	\end{proof}
	
	The next core a priori estimate can be shown following the proof of Proposition \ref{PropKonvSubkrit}. One only has to replace Lemmas \ref{LemLokF} and \ref{LemStab} by Lemmas \ref{LemLokFKorr} and \ref{LemStabKorr}. We thus omit the proof.
	
	\begin{Proposition}\label{PropKonvKorr}
		Let $U=(u,\partial_tu)$, $T$, and $M$ be given by Assumption \ref{Ass} with $\alpha=3$ and let $K=\tau^{-3/2}$. Define the error $E_n$ by \eqref{KorrLie} with $K=\tau^{-3/2}$ and \eqref{DefErrorKorr}.
		Then there is a number $\tau_0>0$ depending only on $M$ and $T$, such that we have the estimate
		\[\opnorm{E_n}_{\tau,\lfloor\frac{T}{\tau}\rfloor} \lesssim_{M,T} 1 \]
		for all $\tau \in (0,\tau_0]$.
	\end{Proposition}
	
	\begin{proof}[Proof of Theorem \ref{ThmKorr}]
		We take $\tau_0>0$ from Proposition \ref{PropKonvKorr} and infer
		\begin{align*}
		\|&U(t_n)-U_n\|_{L^2 \times \dot H^{-1}}\\ 
		&\le \|U(t_n)-\Pi_{K}U(t_n)\|_{L^2 \times \dot H^{-1}} + \|\Pi_{K}U(t_n)-U_n\|_{L^2 \times \dot H^{-1}}\\ 
		&\le \tau^{\frac32}\|U(t_n)\|_{\dot H^1 \times L^2} + \|E_n\|_{L^2 \times \dot H^{-1}} \lesssim_{M,T} \tau^{\frac32}
		\end{align*}
		for $\tau \le \tau_0$, using Lemma \ref{LemPi} and Proposition \ref{PropKonvKorr} with $K=\tau^{-3/2}$.
	\end{proof}

	\printbibliography

\end{document}